\documentclass[11pt,reqno]{amsart}
\newcommand{\R}{{\mathbb R}}
\newcommand{\Pt}{{\widetilde{P}}}

\usepackage[numbers]{natbib}		% Citações padrão ABNT
 \usepackage{amsmath}			% facilidades matemáticas
 \usepackage{amssymb}			% para os símbolos mais antigos
 \usepackage{amstext}			% para fragmentos tipo texto em modo matemático
\usepackage{amsthm}			% para teoremas
\usepackage[pdftex,bookmarksopen=true,colorlinks=true, linkcolor=red, citecolor=blue]{hyperref}
\numberwithin{equation}{section}
\newtheorem{theorem}{Theorem}[section]
\newtheorem{proposition}[theorem]{Proposition}
\newtheorem{lemma}[theorem]{Lemma}
\newtheorem{corollary}[theorem]{Corollary}
\newtheorem{definition}[theorem]{Definition}

\newtheorem{remark}[theorem]{Remark}
\begin{document}
%\citebrackets[]	
%\selectlanguage{english}
\title[Global well-posedness for a coupled system of gKdV equations]{Solitary wave solutions and global well-posedness for a coupled system of gKdV equations}

 % Delete if not wanted.
\author{Andressa Gomes}
\address{IMECC-Unicamp, Rua S\'ergio Buarque de Holanda, 651, CEP 13083-859, Campinas-SP, Brazil}
\email{gomes.andressa.mat@outlook.com}
%\author{Mahendra Panthee}
%\address{Instituto de Matemática, Estatística e Computação Científica, Universidade Estadual de Campinas,
%Campinas, SP}
%\email{mpanthee@ime.unicamp.br}

\author{Ademir Pastor}
\address{IMECC-Unicamp, Rua S\'ergio Buarque de Holanda, 651, CEP 13083-859, Campinas-SP, Brazil}
\email{apastor@ime.unicamp.br}

\begin{abstract}
In this work we consider the initial-value problem associated with a coupled system of generalized Korteweg-de Vries equations. We present a relationship between the best constant for a Gagliardo-Nirenberg type  inequality and a criterion for the existence of global solutions  in the energy space. We prove that such a constant is directly related to the existence problem of solitary-wave solutions with minimal mass, the so called \textit{ground state} solutions. To guarantee the existence of {ground states} we use a variational method.
\end{abstract}

\subjclass[2010]{76B25, 35Q51, 35Q53, 49J27.}

% keywords can be removed
\keywords{Global well-posedness, Coupled KdV systems, Ground-state solutions}
\maketitle
%\tableofcontents

\section{Introduction}
Nonlinear dispersive systems appear in many physical applications. They can be used, for instance, to model the propagation of waves in water surface or to describe the interaction of nonlinear internal waves. In the present paper we are interested in systems having the Hamiltonian form
\begin{equation}\label{systemint}
\begin{cases}
\partial_{t}u + \partial_{x}^{3}u+ \mu \partial_{x}( H_{u}(u,v)) =0,\\
\partial_{t}v + \partial_{x}^{3}v + \mu \partial_{x} (H_{v}(u,v)) =0, 
\end{cases}
\end{equation}
where $u=u(x,t)$ and $v=v(x,t)$ are real-valued functions, $H$ is a smooth function, $H_{u}$ and $H_{v}$ denote the derivative of $H$ with respect to $u$ and $v$, respectively, and $\mu$ is real constant which we normalize to be $\pm1$.

Systems of the form \eqref{systemint} are said to be of KdV type and model important phenomena in the propagation on nonlinear waves. To cite a few examples, in the case $\mu=1$ and
\begin{equation}\label{geargrinshawsytem}
H(u,v) = A u^{3} + Bv^{3} + Cu^{2}v + Duv^{2}
\end{equation}
with $A$, $B$, $C$ and $D$ real constants, the system was derived by Gear and Grimshaw \cite{geargrimshaw} to  describe the strong interaction of two-dimensional long internal gravity waves propagating on neighboring pycnoclines in a stratified fluid.  Also, in the case 
$$H(u,v)=u^2v,$$
 system \eqref{systemint} is a particular case of the Majda-Biello system \cite{majdabiello03} (see also \cite{majda2004} and \cite{biello2009}), which models the nonlinear interaction of long-wavelength equatorial Rossby waves and barotropic Rossby waves.

The issue of local and global well-posedness for the initial-value problem (IVP) associated to \eqref{systemint} have became a major topic in the theory of dispersive equations in recent years. Let us briefly recall some  results of our interest   available in the current literature.
The well-posedness problem associated to IVP \eqref{systemint} with function  $H$ given by \eqref{geargrinshawsytem} was studied by many authors. For instance, Bona, Ponce, Saut and Tom  \cite{bonaponcesauttom}  proved that, under some restrictions on the coefficients, the associated IVP is globally well-posed in  $H^{s}(\mathbb{R}) \times H^{s}(\mathbb{R})$, $s \geq 1$. Also, Linares and Panthee in \cite{pantheelinares04}  obtained  the \textit{sharp} local result for Sobolev spaces with index $s > - \frac{3}{4}$. Besides, in \cite{pantheelinares04} was also proved the global well-posedness for $s > - \frac{3}{10}$ under some restrictions on the coefficients $A$, $B$, $C$ e $D$. The well-posedness for the Majda-Biello system was studied, for instance,  by Oh \cite{oh2009} where the author proved local well-posedness in $H^{s}(\mathbb{R}) \times H^{s}(\mathbb{R})$, $s > -\frac{3}{4}$ and in $H^{-\frac{1}{2}}(\mathbb{T}) \times H^{-\frac{1}{2}}(\mathbb{T})$. In \cite{oh2009-2},  via the I-method, Oh established the global well-posedness
$H^{s}(\mathbb{R}) \times H^{s}(\mathbb{R})$, $s > \frac{3}{4}$ and $H^{s}(\mathbb{T}) \times H^{s}(\mathbb{T})$, $s > -\frac{1}{2}$. Also, Guo \textit{et al.} \cite{titi}  considered the periodic problem and used a successive
time-averaging method to prove the global well-posedness in the homogeneous Sobolev space $\dot{H}^s(\mathbb{T})$, $s\geq0$.

Panthee and Scialom in \cite{pantheescialon} studied \eqref{systemint} with $H(u,v) = \frac{1}{3} u^{3}v^{3}$. In this case, the system contains a pair of ``critical'' generalized KdV equations.   The authors showed local well-posedness in  $H^{s}(\mathbb{R}) \times H^{s}(\mathbb{R})$, $s \geq 0$  utilizing the sharp smoothing estimates to the linear problem combined with the contraction mapping principle. Global well-posedness for data with small Sobolev norm was also established. In particular, they showed if $\|(u_{0}, v_{0})\|_{L^{2} \times L^{2}} < \|(S,S)\|_{L^{2} \times L^{2}}$, where $S$ is an associated ground-sate solution, then the IVP is globally well-posed in $H^{s}(\mathbb{R}) \times H^{s}(\mathbb{R})$ for $s > \frac{3}{4}$.

Corcho and Panthee in \cite{phantheecorcho} considered a  coupled system of modified KdV equations. More precisely, they studied \eqref{systemint} with   
$$
H(u,v) = \frac{a_{1} u^{4}}{4} + \frac{b_{1}}{4}v^{4} + \frac{a^{2}}{2}(uv)^{2} +  \frac{a_{3}}{3}u^{3}v + \frac{a_{4}}{3}uv^{3}.
$$
The authors used the second generations of the modified energy and almost conserved quantities introduced by Colliander, Keel, Staffilani, Takaoka, and Tao \cite{colliander01,  colliander03} to obtain global well-posedness in  $H^{s}(\mathbb{R}) \times H^{s}(\mathbb{R})$ for $s > \frac{1}{4}$.

Alarcon, Angulo and Montenegro \cite{aam} studied  \eqref{systemint} with  $H(u,v) = u^{k+1}v^{k+1}$, where $k \geq 1$ is a natural number and obtained global well-posedness in  $H^{s}(\mathbb{R}) \times H^{s}(\mathbb{R})$, $s\geq1$, under suitable conditions on $k$. Moreover, the authors also established sufficient conditions for the orbital stability and instability  of the associated traveling waves.

Our main objective in this paper is to study the IVP associated with \eqref{systemint} when $H$ has the form
\begin{equation} \label{functionHint}
H(u,v)  = \frac{a}{2k+2} \left( u^{2k+2} + v^{2k+2}\right)+ \frac{b}{k+1} (uv)^{k+1} + \frac{c}{k}u^{k+2}v^{k} + \frac{d}{k}  u^{k}v^{k+2}, \\
\end{equation}
with $k\geq1$ a natural number and $a$, $b$, $c$ e $d$ nonnegative real constants. More precisely, we are interested in the IVP
\begin{equation}\label{PVIsystem}
\begin{cases}
\partial_{t}u + \partial_{x}^{3}u  + \mu \partial_{x} ( f(u,v)) = 0, \\
\partial_{t}v + \partial_{x}^{3}v  + \mu \partial_{x} (g(u,v)) = 0, \ t >0, \ x \in \mathbb{R},\\
(u(x,0), v(x,0) )  =  (u_{0}(x),v_{0}(x)),
\end{cases}
\end{equation}
with
\begin{equation}\label{nonlinear}
\begin{cases}
f(u,v) := H_u(u,v)= a\  u^{2k+1} + b \ u^{k}v^{k+1} + \frac{k+2}{k} \ c \ u^{k+1}  v^{k} + d \  u^{k-1}v^{k+2}, \\
g(u,v) := H_v(u,v)= a \ v^{2k+1} + b \ v^{k}u^{k+1} +  \frac{k+2}{k} \ d \ v^{k+1} u^{k} +c \  v^{k-1}u^{k+2}.
\end{cases}
\end{equation}
From \eqref{functionHint} and \eqref{nonlinear} it is easily seen that
\begin{equation}\label{functionh2}
H(u,v)=\frac{1}{2k+2}\left[f(u,v)u+g(u,v)v\right]=\frac{1}{2k+2}\left[H_u(u,v)u+H_v(u,v)v\right].
\end{equation}
Following the standard nomenclature in the literature, for $\mu =1$ the system \eqref{PVIsystem} is said to be \textit{focusing} whereas for $\mu = - 1$ it is called \textit{defocusing}. Note that our function  $H$ given by \eqref{functionHint} generalizes the models in \cite{aam, bona17, bonaponcesauttom, phantheecorcho, pantheelinares04, pantheescialon}.  So our work may be seen as a natural extension of these works.

Let us now describe our results. First of all, the local well-posedness for IVP \eqref{PVIsystem} can be established similarly to \cite{aam}. More specifically, combining smoothing effects with a contraction principle argument we obtain the following result.

\begin{theorem}\label{lwpsystem}
Let $k\geq 1$ and $s \geq 1$. Then for any $(u_{0}, v_{0}) \in H^{s}(\mathbb{R}) \times H^{s}(\mathbb{R})$ there exists $T=T(\|(u_{0},v_{0})\|_{H^s}) > 0$ and a unique strong solution $(u,v)$ of the IVP \eqref{PVIsystem} in the class
\begin{equation}\label{solutionclass}
\begin{array}{ccc}
(u,v) \in C([0,T] ; H^{s}(\mathbb{R}) \times H^{s}(\mathbb{R})),\\ 
\|(u,v)\|_{L_{x}^{2}L_{T}^{\infty} \times L_{x}^{2}L_{T}^{\infty}}  < \infty,\\ %smothing effect
\|(u_{x},v_{x})\|_{L_{T}^{4}L_{x}^{\infty} \times L_{T}^{4}L_{x}^{\infty}} < \infty,\\% strichartz estimate
\|\partial_{x} D_{x}^{s}(u,v)\|_{L_{x}^{\infty}L_{T}^{2} \times L_{x}^{\infty}L_{T}^{2}} < \infty .\\ % maximal function
\end{array}
\end{equation}
Moreover, for any $T_{0} \in (0,T)$ there exists a neighborhood $V_{0}$ of $(u_{0}, v_{0}) \in H^{s}(\mathbb{R}) \times H^{s}(\mathbb{R})$ such that the map $(u_{0}, v_{0}) \mapsto (u(t), v(t))$ from $V_{0}$ into the class defined in \eqref{solutionclass} (with $T_0$ instead of $T$) is Lipschitz.
\end{theorem}

By noting that  $f$ and $g$ are homogeneous polynomials of degree $2k+1$, the proof of Theorem \ref{lwpsystem} is similar to that of Theorem 3.1 in \cite{aam}. So we will omit the details.
Once  we know the existence of local solutions, a natural question is about their extension to global ones.
This question is partially answered for solutions in the energy space $H^{1}(\mathbb{R}) \times H^{1}(\mathbb{R})$ in view of the conservation laws. Indeed, it is not difficult to see that system \eqref{systemint}  conserves  the mass and the energy  given, respectively, by
\begin{equation}\label{laws3} 
M (u ,v) = \displaystyle \frac{1}{2} \int_{\mathbb{R}} [u^{2} + v^{2}] dx 
\end{equation}
and
\begin{equation}\label{laws4} 
E (u,v) = \displaystyle \frac{1}{2} \int_{\mathbb{R}} [(\partial_{x}u)^{2} + (\partial_{x}v)^{2} - 2 \mu H(u,v)] dx.
\end{equation}
In addition, since the existence time in Theorem \ref{lwpsystem} depends on the norm of the initial data itself, in order to extend the solution globally-in-time it suffices to establish an a priori bound on $\|\partial_{x}(u, v)(t)\|$, where $\|(\cdot, \cdot)\|$ denotes the $L^{2}(\mathbb{R}) \times L^{2}(\mathbb{R})$-norm. 
Observe that \eqref{laws4} provides
\begin{equation}\label{idderivativeaux} 
\|\partial_{x}(u, v)(t)\|^{2} = 2E(u_{0}, v_{0}) + 2 \mu \int H(u,v)(t) dx.
\end{equation}
As an immediate consequence, in the case   $\displaystyle \mu\int H(u,v) dx < 0$, we have the following result.

\begin{proposition}\label{gwpspecialcase}
Let $(u_{0}, v_{0}) \in  H^{1}(\mathbb{R}) \times H^{1}(\mathbb{R})$. If  $\displaystyle \mu\int H(u,v)dx < 0$ then there exists a unique  solution $(u, v)$ of IVP \eqref{PVIsystem} satisfying
\begin{equation*}
    (u,v) \in \mathcal{C}(\mathbb{R} : H^{1}(\mathbb{R}) \times H^{1}(\mathbb{R})) \cap L^{\infty}(\mathbb{R} : H^{1}(\mathbb{R}) \times H^{1}(\mathbb{R})).
\end{equation*}
\end{proposition}

\begin{remark}
Assume $\mu=-1$. When $k$ is an even number  and $b$ is null or when  $k$ is  odd and $c=d=0$ we have $\displaystyle \int H(u,v)dx > 0$. In particular, in these cases we see that assumption in Proposition \ref{gwpspecialcase} is fulfilled.
\end{remark}

On the other hand, from Sobolev's embedding and Cauchy-Schwartz's inequality, the following estimate hold
\begin{align}\label{inequalityGNtypeint}
2 \int H(|u|,|v|) dx  \leq C  \|(u,v)\|^{k+2}\|\partial_{x}(u,v)\|^{k},
\end{align}
where $C$ is a positive constant. So, in view of \eqref{idderivativeaux},
$$
\|\partial_{x}(u,v)\|^{2} \leq 2 E(u_{0},v_{0}) + C \|(u,v)\|^{k+2}\|\partial_{x}(u,v)\|^{k}.
$$
Hence, by using a standard argument (see, for instance, \cite[Chapter 6]{linaresponce})  we can establish the existence of global solutions for \eqref{PVIsystem} under certain conditions. More precisely,

\begin{proposition}\label{gwpsystemh1}
Let $(u_{0}, v_{0}) \in H^{1}(\mathbb{R}) \times H^{1}(\mathbb{R})$. Then the solution $(u,v)$ given by Theorem \ref{lwpsystem} can be extended to any interval $[0,T]$,  $T > 0$, under one of the following assumptions:
\begin{enumerate}
\item[(i)]  $k=1$ and no restrictions on the initial data.
\item[(ii)]  $k=2$ and $\|(u_{0}, v_{0})\|$  small enough.
\item[(iii)]  $k>2$ and $\|(u_{0}, v_{0})\|_{H^{1} \times H^{1}}$ small enough.
\end{enumerate}
\end{proposition}

 Proposition \ref{gwpsystemh1} is in agreement with the result in \cite[Theorem 4.1]{aam}. Note that in the case $k\geq2$ we always need a smallness assumption on the initial data. As is well-known this is a feature of $L^2$ supercritical dispersive equations. In this paper, our main contribution is to give a more precise description of how small the initial data must be.

Our main result reads as follows (for the precise definition of ground states see Definition \ref{groundstate})

\begin{theorem}[Global well-posedness in $H^{1}(\mathbb{R}) \times H^{1}(\mathbb{R})$]\label{teogwpsystem}
Let $(u_{0},v_{0}) \in H^{1}(\mathbb{R}) \times H^{1}(\mathbb{R}) $ and $k \geq 2$. Suppose that 
\begin{equation}\label{hip1}
M(u_{0},v_{0})^{k+2} E(u_{0},v_{0})^{k-2} < M(\Phi ,\Psi)^{k+2} E(\Phi , \Psi)^{k-2},
\end{equation}
where $(\Phi,\Psi)$ is a \textit{ground-state} solution of the elliptic system 
\begin{equation}\label{sistemaeliptico2}
\begin{cases}
\phi'' - \phi + f(\phi, \psi )= 0,\\
\psi '' - \psi +  g(\phi, \psi )= 0.
\end{cases}
\end{equation} 
 If 
\begin{equation}\label{hip2}
\|\partial_{x}(u_{0}, v_{0})\|^{k-2}\|(u_{0},v_{0})\|^{k+2} < \|\partial_{x}(\Phi, \Psi)\|^{k-2}\|(\Phi ,\Psi)\|^{k+2},
\end{equation}
then as long as the local solution given in Theorem \ref{lwpsystem} exists, it satisfies
\begin{equation}\label{res2}
\|\partial_{x}(u(t), v(t))\|^{k-2}\|(u_{0},v_{0})\|^{k+2} < \|(\partial_{x} (\Phi,  \Psi)\|^{k-2}\|(\Phi ,\Psi)\|^{k+2}.
\end{equation}
In particular, the solution exists globally-in-time in $H^{1}(\mathbb{R}) \times H^{1}(\mathbb{R})$.
\end{theorem}

\begin{remark}
Note that in the case $k=2$ assumptions \eqref{hip1} and \eqref{hip2} reduce to the same one and are equivalent to $\|(u_0,v_0)\|<\|(\Phi,\Psi)\|$. In the case $k>2$,  it must be understood that the energy $E(\Phi,\Psi)$ appearing in \eqref{hip1} is evaluated for $\mu=1$.
\end{remark}

To prove Theorem \ref{teogwpsystem},  we first relate the best constant one can place in inequality \eqref{inequalityGNtypeint} with the problem of existence of \textit{ground state} solutions associated to \eqref{sistemaeliptico2} (see Corollary \ref{teosharpconst}). The main idea is to see the ground states as minima of a Weinstein-type functional. Now a day, many strategies can be used to obtain the minima of such a functional.  Here, we use the ones adopted, for instance in Maia, Montefusco and Pellaci  \cite{ mmontefusco06}, Fanelli and Montefusco  \cite{fanellim2007}, Pastor \cite{pastor15}, Hayashi, Ozawa and Tanaka  \cite{ozawa13}, Noguera and Pastor \cite{normanp2018}, where the authors  established the existence of ground states for coupled nonlinear Schr\"odinger equations and presented sufficient conditions for the global existence related with those equations. 
We also refer to the work of Esfahani and Pastor in \cite{esfahani18}, where the authors studied a generalized Shrira equation.

By setting $v=0$ in \eqref{PVIsystem} we see that the system reduces to the generalized KdV equation
\begin{equation}\label{gkdv}
\partial_{t}u + \partial_{x}^{3}u  + \mu a\partial_{x}(u^{2k+1})=0.
\end{equation}
Equation \eqref{gkdv} together with the Schr\"odinger equation are the most studied dispersive models. Many results concerning local and global well-posedness, asymptotic behavior, and several other properties of the solutions can be found in the current literature, which we refrain from list them at this stage. However, a similar result for \eqref{gkdv} as the one in Theorem \ref{teogwpsystem} was established in \cite{flpastor11}. So, Theorem \ref{teogwpsystem} may also be seen as an extension to that result for system \eqref{PVIsystem}.

As our second main result, we give a suitable characterization of the ground states. Indeed, as we will see in Section \ref{GNsec}, if $(\Phi,\Psi)$ is a  ground state of \eqref{sistemaeliptico2}  then it is nonnegative, that is, $\Phi\geq0$, $\Psi\geq0$. In addition, since the coefficients of $f$ and $g$ are nonnegative,
\begin{equation*}
\begin{cases}
\Phi'' - \Phi = - f(\Phi, \Psi) \leq 0,\\
\Psi'' - \Psi = - g(\Phi, \Psi) \leq 0.
\end{cases}
\end{equation*}
By the maximum principle (see \cite[Theorem 3.5]{gilbarg}) it follows that $\Phi$ is strictly positive or vanishes everywhere. A similar statement holds for $\Psi$. If $\Psi\equiv0$, for instance, then $\Phi$ is a solution of the scalar equation $\Phi''-\Phi+f(\Phi,0)=0$. A natural and interesting question is when the ground states are of the form $(\Phi,0)$ or $(0,\Psi)$, which we will pay particular attention below.

It is well known that equation
\begin{equation}\label{scalarg}
-Q''+Q-Q^{2k+1}=0,
\end{equation}
has a unique ground state (up to translations), which is positive, radially symmetric and has an exponential decay at infinity  (see, for instance, \cite[Chapter 8]{CAZENAVEBOOK}).

Our main result here is the follows.

\begin{theorem}\label{char}
	Let $F(x,y)=(2k+2)H(x,y)$, where $H$ is given in \eqref{functionHint}. Let $Y$ be the set of all points $(x_0,y_0)$ satisfying
	$$
	F(x_0,y_0)=F_{max}:=\max\{F(x,y): \,x^2+y^2=1, x\geq0, y\geq0\}.
	$$
	A pair $(u,v)\in H^1(\R)\times H^1(\R)$ is a nonnegative ground state of \eqref{sistemaeliptico2} if and only if the exists $\alpha,\beta\geq0$ such that $(F_{max})^{\frac{1}{2k}}(\alpha,\beta)\in Y$ and
	$$
	(u,v)=(\alpha Q,\beta Q),
	$$
	where  $Q$ is the ground state of  \eqref{scalarg}.
	
	In particular, uniqueness of the ground states holds provided $Y$ has only one point.
\end{theorem}

The proof Theorem \ref{char} relies on an extension of the arguments in \cite{correia}. In particular it heavily depends on the fact that $H$ is an homogeneous function. Thus, the ground states can also be viewed as minima of another suitable minimization problem.

Observe that system \eqref{sistemaeliptico2} appears when we look for solitary waves (with velocity one) of system in \eqref{PVIsystem} with $\mu=1$. Indeed, a solitary-wave solution of \eqref{PVIsystem} is a solution having the form $u(x,t)=\phi(x-\omega t)$, $v(x,t)=\psi(x-\omega t)$, where $\omega$ is a real constant representing the velocity of the traveling wave. By substituting this form in \eqref{PVIsystem}, with $\omega=1$, we promptly see that $(\phi,\psi)$ must satisfy \eqref{sistemaeliptico2}.  Another question of our interest here concerns the orbital stability/instability of the solitary waves. As we will see in Section \ref{secinst} at least under some restrictions on the coefficients appearing in the definition of $H$ we are able to establish the orbital instability.

The paper is organized as follows. In Section \ref{preli} we introduce some notations and recall some standard results which we use along the paper. In Section \ref{GNsec} we prove the existence of ground state solutions associated with system \eqref{sistemaeliptico2}. As a consequence we also obtain a sharp Gagliardo-Nirenberg inequality. Section \ref{globalsec} is devoted to prove Theorem \ref{teogwpsystem}. In Section \ref{secchar} we give our characterization of the ground states by proving Theorem \ref{char}. Finally, in Section \ref{secinst} we establish our instability result.

\section{Notation and preliminaries}\label{preli}
In this section we list some  notation that will be used in this work. We also recall some basic results that will be used along the paper. Given a measurable set $\Omega\subset\R^n$, $|\Omega|$ denotes its Lebesgue measure.  Given a function $f$ and a number $\lambda>0$, the sets $\{x \in \mathbb{R}^n : f(x) \neq 0 \}$ and $\{x\in\R^n : |f(x)|>\lambda\}$ will be denoted, respectively, by $\{f\neq0\}$ and $\{|f|>\lambda\}$.

The standard Lebesgue spaces will be denoted by $L^p(\R^n)$, $1\leq p\leq +\infty$.
For $s\in\R$, by $H^{s}(\mathbb{R}^n)$ we denote the $L^{2}$-based Sobolev space of order $s$ with norm
\begin{equation*}
\|f\|_{H^{s}}= \| \left< \xi \right>^{s} \widehat{f}\|_{L^{2}},
\end{equation*}
where $\left< \xi \right> =  1 + |\xi|$ and $\widehat{f}$ denotes the Fourier transform of $f=f(x)$.  For $s\in\R$, the operator $D^s_x$ is defined via its Fourier transform as $\widehat{D^s_xf}(\xi)=|\xi|^s\widehat{f}(\xi)$. To simplify notation, we use $\|\cdot\|$ to denote the norm in $L^2(\R)\times L^2(\R)$, that is, $\|(u,v)\|^2=\|u\|_{L^2}^2+\|v\|_{L^2}^2$. The notation $\int fdx$ always means $\int_{\R}f(x)dx$. In general $C$ denotes a constant that may vary from one inequality to another.

Now, we give some results necessary for future statements. These results are not new and can be found in the current literature. As we will see below our arguments to prove the existence of ground states will be based on the Mountain Pass Theorem without the Palais-Smale condition which reads as follows.

\begin{theorem}[Mantain Pass Theorem]\label{PSseq}
Let $Y$ be a Hilbert space and $\varphi \in C^{2}(Y, \mathbb{R})$. If there exist $\widetilde{u} \in Y$ and $r>0$ such that $\|\widetilde{u}\|_{Y} > r$ and
$$
\displaystyle \sigma := \inf_{\|u\|_{Y} = r} \varphi(u) > \varphi(0) \geq \varphi(\widetilde{u})
$$ 
then there exists a sequence $(u_{n}) \subset Y$ satisfying
\begin{flalign}
\label{PS1}  & \ \varphi(u_{n}) \rightarrow \omega,&\\
\label{PS2}  & \ \varphi '(u_{n}) \rightarrow 0 \ \ \mbox{strongly in } \ \ Y', &
\end{flalign}
where  $ \displaystyle \omega = \inf_{\gamma \in \Gamma} \max_{t \in[0,1]} \varphi (\gamma (t) )$ and
$$
\displaystyle \Gamma = \{ \gamma \in C([0,1], Y) \, : \, \gamma(0) =0 \,\, \mbox{and} \,\, \varphi(\gamma(1)) < 0 \}.
$$
\end{theorem}
\begin{proof}  See Theorem 1.15 in \cite{willen}.
\end{proof}

A sequence $(u_n)$ satisfying \eqref{PS1} and \eqref{PS2} will be called a \textit{$(PS)_{\omega}$-sequence} for the functional $\varphi$.  Next we recall two important inequalities we will use below. 

\begin{proposition}[Faber-Krahn's inequality] \label{faber}
Assume that $u \in H^{1}(\mathbb{R}^{n})$ satisfies $0 <  \left| \{  u  \neq   0 \} \right| < \infty.$
Then, there is $C  >  0$ such that
\begin{equation*}%\label{faberkrahn}
\|u\|_{L^{2}(\mathbb{R}^{n})}^{2} \leq C |\{ u   \neq  0\}|^{\frac{2}{n}}\|\nabla u\|_{L^{2}(\mathbb{R}^{n})}^{2}.
\end{equation*}
\end{proposition}
\begin{proof}
From H\"older and Gagliardo-Nirenberg's inequalities, for any $q>2$,
\[
\begin{split}
\|u\|_{L^2}^2&\leq |\{u\neq0\}|^{\frac{q-2}{q}}\|u\|_{L^q}^2\\
& \leq C|\{u\neq0\}|^{\frac{q-2}{q}}\|u\|_{L^2}^{2(1-\theta)}\|\nabla u\|_{L^2}^{2\theta}, \qquad \theta=n\left(\frac{1}{2}-\frac{1}{q}\right).
\end{split}\]
Since $\frac{q-2}{q}=\frac{2\theta}{n}$ the result then follows.
\end{proof}

\begin{proposition}[Chebyshev's inequality]\label{chebyshev}
If $0 < p < \infty$, then for any $\lambda>0$,
$$
\|f\|_{L^{p}}^{p} \geq \lambda^{p} | \{|f(x)| > \lambda \}|.
$$
\end{proposition}
\begin{proof}
See Theorem 6.17 in \cite{folland1}.
\end{proof}

The proof of Theorem \ref{teogwpsystem} will be based on a continuity argument. To simplify the exposition we recall the following.

\begin{lemma}\label{lemacontinuidade}
	Let $I  \subset \mathbb{R}$ an open interval containing $0$. Let $m >1$, $B > 0$ and $A$ be  real constants. Define $\gamma = (Bm)^{-\frac{1}{m-1}}$ and $f(r) = A- r + Br^{m}$, for $r \geq 0$. Let $G(t)$ be a continuous nonnegative function on $I$. Assume that $A < \left( 1- \frac{1}{m}\right)\gamma$ and $f\circ G \geq 0$.
	\begin{itemize}
		\item [(i)] If $G(0) < \gamma$, then $G(t) < \gamma $, for any $t \in I$.
		\item [(ii)] If $G(0) > \gamma$, then $G(t) > \gamma$, for any $t \in I$.
	\end{itemize}
\end{lemma}
\begin{proof}
	See Lemma 3.1 in \cite{pastor15}.
\end{proof}

\section{Gagliardo-Nirenberg type  inequality and ground states}\label{GNsec}
As we pointed out above, the proof of Proposition \ref{gwpsystemh1} is an immediate consequence of the Gagliardo-Nirenberg type  inequality \eqref{inequalityGNtypeint} and a standard argument. In addition, it is clear that the smallness assumption in Proposition \ref{gwpsystemh1} is related to the constant appearing in \eqref{inequalityGNtypeint}. Hence, the main goal of this section is to study the best constant one can place in \eqref{inequalityGNtypeint}. From now on we assume $\mu = 1$. 

Let us start by introducing the set
\begin{equation}\label{sethalmitoniano}
\mathcal{P} = \{ (u,v) \in  H^{1}(\mathbb{R}) \times H^{1}(\mathbb{R})\setminus \{(0,0)\}  \, : \, P(u,v):=\int H(u,v)dx > 0  \}
\end{equation}
and the functional
\begin{equation}\label{functionalJ}
J(u,v) = \displaystyle \frac{\|(u,v)\|^{k+2} \|\partial_{x}(u, v)\|^{k}}{2  \displaystyle \int H(u,v) dx}.
\end{equation}

\begin{remark}
We always have $\mathcal{P}\neq\emptyset$. Indeed, for any $u\in H^1(\R)\setminus\{0\}$ we obtain 
$$
\int H(u,u)=\left(\frac{a+b}{k+1}+\frac{c+d}{k}\right)\int u^{2k+2}dx>0,
$$
which means that $(u,u)\in\mathcal{P}$.
\end{remark}

From \eqref{inequalityGNtypeint} we immediately see that, on $\mathcal{P}$, functional $J$ is bounded from below by a positive constant. As a consequence, the best constant we can place in \eqref{inequalityGNtypeint} is $K_{opt}$ given by
\begin{equation}\label{konstantoptima}
K_{opt}^{-1} =  \displaystyle \inf \{ J(u,v) \, : \, (u,v) \in\mathcal{P}\}.
\end{equation}
So, our task is to understand the infimum of $J$ on the set $\mathcal{P}$. As we will see below such a infimum is attained in a special solution of \eqref{sistemaeliptico2}.

\begin{definition}\label{defweaksol}
The pair $(\phi,\psi) \in H^{1}(\mathbb{R}) \times H^{1}(\mathbb{R})$ is said to be  a (weak) solution of \eqref{sistemaeliptico2} if 
\begin{equation}\label{weaksol}
\begin{cases}
\displaystyle \int \phi w dx + \int \phi'w' dx  - \int f(\phi,\psi)w dx = 0,\\
\displaystyle \int \psi z dx  + \int \psi'z' dx  -  \int g(\phi, \psi)z dx = 0,
\end{cases}
\end{equation}
for any  $(w,z) \in H^{1}(\mathbb{R}) \times H^{1}(\mathbb{R})$.
\end{definition}

It is not difficult to see that $(\phi,\psi)$ is a solution of \eqref{sistemaeliptico2} if and only if it is a critical point of the action functional
	\begin{equation}\label{operatorI}
	I(u,v) := M(u,v) + E(u,v) = \displaystyle \frac{1}{2}\|(u,v)\|^{2} + \frac{1}{2}\|\partial_{x}(u,v)\|^{2} - P(u,v).
	\end{equation}
In addition, by the standard elliptic regularity theory any weak solution is indeed smooth and can be regarded as a solution in the strong sense (see, for instance, \cite[Chapter 8]{CAZENAVEBOOK}).
Among all critical points of \eqref{operatorI}, the minima play a distinguished role  in several aspects of \eqref{sistemaeliptico2}; they are called ground states.

\begin{definition}\label{groundstate}
	A pair of real-valued functions $(\phi, \psi) \in H^{1}(\mathbb{R}) \times H^{1}(\mathbb{R})$ is called a \textit{ground-state} solution of \eqref{sistemaeliptico2} if
	$$
	I(\phi, \psi) = \inf \{ I(u,v) \, : \, (u,v) \in H^{1}(\mathbb{R}) \times H^{1}(\mathbb{R})\setminus {(0,0)} \ \ \mbox{and} \ \ I'(u,v) =0 \}.
	$$
\end{definition}

Next we give some properties of the solutions of \eqref{sistemaeliptico2}.

\begin{proposition}[Pohozaev type identities]\label{Pohozaevid}
Let $(\phi, \psi)$ be a solution of \eqref{sistemaeliptico2}. The following identities hold.
\begin{flalign}
 \label{Pohosaevid1} \mathrm{(i)} & \ \ \displaystyle \|(\phi, \psi)\|^{2} + \|\partial_{x}(\phi, \psi)\|^{2} =  (2k+2)P(\phi, \psi) ;&\\
\label{Pohosaevid2} \mathrm{(ii)}  &\ \ \displaystyle \|(\phi, \psi)\|^{2} - \|\partial_{x}(\phi, \psi)\|^{2} = 2 P(\phi, \psi) ;&\\
\label{Pohosaevid3} \mathrm{(iii)} & \ \ \displaystyle \|\partial_{x}(\phi, \psi)\|^{2} = \frac{k}{k+2}\|(\phi, \psi)\|^{2};&\\
\label{Pohosaevid4} \mathrm{(iv)} & \ \ \displaystyle P(\phi, \psi)  = \frac{1}{k+2}\|(\phi, \psi)\|^{2};&\\
\label{Pohosaevid5} \mathrm{(v)}  & \ \  \displaystyle P(\phi, \psi) = \frac{1}{k}\|\partial_{x}(\phi, \psi)\|^{2}.&
\end{flalign}
In particular, any nontrivial solution of \eqref{sistemaeliptico2} belongs to $\mathcal{P}$.
\end{proposition}
\begin{proof}
By taking $(w,z) = (\phi, \psi)$ in \eqref{weaksol}, we obtain
\begin{equation}
\displaystyle \|(\phi, \psi)\|^{2} + \|\partial_{x}(\phi,\psi)\|^{2} = \displaystyle  \int_{\mathbb{R}} \left[f(\phi,\psi)\phi+ g(\phi,\psi)\psi \right] dx.
\end{equation}
From  \eqref{functionh2} we conclude the prove of \eqref{Pohosaevid1}.
On the other hand, we show \eqref{Pohosaevid2} by multiplying the equations in (\ref{sistemaeliptico2}) by $x\phi'$ and $x\psi'$, respectively, integrating on the spatial variable and applying integration by parts.  The identity \eqref{Pohosaevid3} results from multiplying \eqref{Pohosaevid2} by $-(k+1)$ and adding to \eqref{Pohosaevid1}. Finally,  identities \eqref{Pohosaevid4} and \eqref{Pohosaevid5} are obtained by
adding and subtracting, respectively, the equations \eqref{Pohosaevid1} and \eqref{Pohosaevid2}.
\end{proof}

The Pohozaev identities allow us to prove the   equivalence between minimizing the functionals  $J$ and $I$.

\begin{proposition}\label{propIandsolution}
If $(\phi,\psi)$ is a solution of \eqref{sistemaeliptico2} then
\begin{align}\label{functionalJapp}
J(\phi, \psi) = \displaystyle \frac{k+2}{2} \left( \frac{k}{k+2}\right)^{\frac{k}{2}}\|(\phi, \psi)\|^{2k}.
\end{align}
and
\begin{equation}\label{Iandsolution}
\displaystyle I(\phi,\psi) = \frac{k}{k+2}\|(\phi,\psi)\|^{2}.
\end{equation}
In particular, a nontrivial solution $(\phi,\psi)\in\mathcal{P}$ of \eqref{sistemaeliptico2} is a minimizer of $J$ if and only if it is a ground state.
\end{proposition}
\begin{proof}
From Proposition \ref{Pohozaevid}, we obtain
\begin{align*}
J(\phi, \psi) &=  \displaystyle \frac{\|(\phi,\psi)\|^{k+2} \left( \frac{k}{k+2} \|(\phi, \psi)\|^{2}\right)^{\frac{k}{2}}}{  \frac{2}{k+2} \|(\phi, \psi)\|^{2}} = \displaystyle \frac{k+2}{2} \left( \frac{k}{k+2}\right)^{\frac{k}{2}}\|(\phi, \psi)\|^{2k}
\end{align*}
and
\begin{equation*}
\begin{split}
I(u,v) 
= \displaystyle \frac{1}{2}\|(u,v)\|^{2} + \frac{1}{2}\left( \frac{k}{k+2} \right)\|(u,v)\|^{2} -\frac{1}{k+2} \|(u,v)\|^{2} = \displaystyle \frac{k}{k+2} \|(u,v)\|^{2},
\end{split}
\end{equation*}
which is the desired.
\end{proof}

\begin{remark}\label{caracterizacaogs}
 Proposition \ref{propIandsolution} ensures that the {ground-state} solutions of \eqref{sistemaeliptico2} are solutions which minimize the $L^{2}(\mathbb{R}) \times L^{2}(\mathbb{R})$ norm.
\end{remark}

The next sections will be dedicated to prove the existence of {ground state} solutions for \eqref{sistemaeliptico2}. Once we do that, we also obtain the minimum of $J$, which is our main goal.

\subsection{Variational theory}\label{sectionteovariacional}

In this section, we use the Mountain Pass Theorem (Theorem \ref{PSseq}) to obtain a sequence which provides a minimum for the functional $I$.

\begin{proposition}\label{existencePS}
The  \textit{action functional} $I$ defined in  \eqref{operatorI} admits a $(PS)_{\omega}$-sequence with $\omega > 0$ given by
\begin{equation}\label{mim3} 
\omega := \displaystyle \inf_{\gamma \in \Gamma} \max_{\ell \in [0,1]} I(\gamma(\ell)),
\end{equation}
where
\begin{equation*}%\label{setGamma}
\Gamma := \{ \gamma \in C([0,1],H^{1}(\mathbb{R}) \times H^{1}(\mathbb{R})) \ \ \mbox{and} \ \ \gamma(0) =(0,0), \ I(\gamma(1)) < 0 \}.
\end{equation*}
\end{proposition}
\begin{proof}
	It suffices to show that $I$ satisfies the mountain pass geometry in Theorem \ref{PSseq}. First of all note that
 inequality \eqref{inequalityGNtypeint} guarantees the existence of $C_{0} >0$ such that
\begin{align}\label{pa1}
I(u,v) = \frac{1}{2}\|(u, v)\|_{H^{1} \times H^{1}}^{2} - P(u,v) \geq \frac{1}{2}\|(u, v)\|_{H^{1} \times H^{1}}^{2} - C_{0} \|(u, v)\|_{H^{1} \times H^{1}}^{2k+2}.
\end{align}
So, for $r$ sufficiently small there exists $\delta_{0} >0$ such that $I(u,v) \geq \delta_{0}$ for all $(u,v) \in H^{1}(\mathbb{R}) \times H^{1}(\mathbb{R})$ sa\-tisfying $ \|(u, v)\|_{H^{1} \times H^{1}} = r$. Furthermore, by continuity, taking  any $\gamma \in \Gamma$, we have $\max_{\ell \in [0,1]}I(\gamma(\ell)) \geq I(u,v)$ for all $(u,v) \in H^{1}(\mathbb{R}) \times H^{1}(\mathbb{R})$ such that $ \|(u, v)\|_{H^{1} \times H^{1}} = r $. Thus, $\omega \geq \delta_{0} >0$.

Now, fix  $(u,v) \in \mathcal{P}$ and set $(\widetilde{u}, \widetilde{v}):= (Lu,Lv) $, where $L>0$ will be chosen conveniently. Thus,
\begin{align*}
I(\widetilde{u}, \widetilde{v})= \frac{L^{2}}{2} \|(u, v)\|_{H^{1} \times H^{1}}^{2} - L^{2k+2} P(u,v),
\end{align*}
By choosing  $L$ sufficiently large, we obtain $\|(\widetilde{u} , \widetilde{v})\|_{H^{1} \times H^{1}} = L \|(u,v) \|_{H^{1} \times H^{1}} > r$ and $I(\widetilde{u}, \widetilde{v}) < 0$. Consequently, 
$$
I(\widetilde{u}, \widetilde{v}) < 0 = I(0,0) < \delta_{0} \leq \inf \ \{I(u,v) \, : \, \|(u,v)\|_{ H^{1} \times H^{1}} =r \}=: \sigma.
$$
The result then follows from the Mountain Pass Theorem.
\end{proof}

Next result gives some additional properties of any $(PS)_{\eta}$-sequence of the functional $I$.

\begin{proposition}\label{convegencederivative}
Let $(u_{n},v_{n})$ be any $(PS)_{\eta}$-sequence of the functional $I$. Then, $(u_{n},v_{n})$ is  bounded in $H^{1}(\mathbb{R}) \times H^{1}( \mathbb{R})$. Moreover,
\begin{flalign}
\label{conder1} \mathrm{(i)} & \ \ \displaystyle \|(u_{n},v_{n})\|_{H^{1} \times H^{1}}^{2} \longrightarrow \frac{2k+2}{k}\eta ;& \\
\label{conder2}  \mathrm{(ii)} & \ \ \displaystyle P(u_{n},v_{n}) \longrightarrow \frac{\eta}{k}.&
\end{flalign}
In particular, $\eta \geq 0$ and $\eta = 0$ if and only if $u_{n} \longrightarrow 0$, $v_{n} \longrightarrow 0$ in $H^{1}(\mathbb{R})$.
\end{proposition} 
\begin{proof} 
	To begin with, note that a simple calculation gives
	\begin{equation}\label{derivativeI}
	\displaystyle I'(u,v)(w,z) = \int(uw + vz) dx + \int (\partial_{x} u \partial_{x} w + \partial_{x} v \partial_{x} z) dx - \int (f(u,v)w + g(u,v)z) dx
	\end{equation}
 So,  by taking $(w,z) = (u_{n}, v_{n})$ and using \eqref{functionh2} we obtain
\begin{equation}\label{conv1}
\displaystyle \left| \|(u_{n},v_{n})\|_{H^{1} \times H^{1}}^{2} - (2k+2) P(u_{n},v_{n}) \right|  = |I'(u_{n},v_{n})(u_{n},v_{n}) |
\leq K_{n} \|(u_{n}, v_{n})\|_{H^{1} \times H^{1}},
\end{equation}
where $K_{n} := \|I'(u_{n},v_{n})\|_{H^{-1} \times H^{-1}} $.
Now, note that
\begin{equation*}%\label{idnormaseq}
\displaystyle \frac{2k+2}{k}  \left( I(u_{n},v_{n}) -\eta \right)  -\frac{1}{k} I'(u_{n},v_{n})(u_{n},v_{n}) + \frac{2k+2}{k} \eta 
= \|(u_{n},v_{n})\|_{H^{1} \times H^{1} }^{2}.
\end{equation*}
Therefore,
\begin{align*}%\label{con2}
\|(u_{n},v_{n})\|_{H^{1} \times H^{1}}^{2} 
& \leq \frac{2k+2}{k} | I(u_{n},v_{n}) -\eta | + \frac{K_{n}}{k} \|(u_{n},v_{n})\|_{H^{1} \times H^{1}} + \frac{2k+2}{k} | \eta  | \\
& \leq \frac{2k+2}{k} |\left( I(u_{n},v_{n}) -\eta \right)| + \frac{K_{n}^{2}}{2k^{2}} +  \frac{1}{2}\|(u_{n},v_{n})\|_{H^{1} \times H^{1}}^{2} + \frac{2k+2}{k} | \eta  | 
\end{align*}
where in the last inequality we used the Young inequality. Since $I(u_n,v_n)\to\eta$ and $K_n\to0$ we deduce that $(u_n,v_n)$ is bounded.

On the other hand, from  \eqref{conv1} we infer
$$
\displaystyle \|(u_{n},v_{n})\|_{H^{1} \times H^{1}}^{2} -  (2k+2) P(u_{n},v_{n}) \longrightarrow 0 .
$$
Hence, in order to conclude the proof  it suffices to prove \eqref{conder1}. But since
\begin{align*}
\left| \|(u_{n},v_{n})\|_{H^{1} \times H^{1}( \mathbb{R})}^{2} -  \frac{2k+2}{k}  \eta  \right|  &  \leq  \displaystyle \frac{2k+2}{k} | \left( I(u_{n},v_{n}) - \eta \right)| + \frac{1}{k} |I'(u_{n},v_{n})(u_{n},v_{n})|  \\
& \leq \frac{2k+2}{k} |\left( I(u_{n},v_{n}) -\eta \right)| + \frac{K_{n}}{k} \|(u_{n},v_{n})\|_{H^{1} \times H^{1}}.
\end{align*}
we obtain the desired.
\end{proof}

\subsection{Compactness}\label{sectionconcentracaocompacidade}

Our goal is to show that, up to a subsequence and a spatial translation, the $(PS)_{\omega}$-sequence obtained in Proposition \ref{existencePS} converges in $H^{1}(\mathbb{R}) \times H^{1}(\mathbb{R})$ to a  function  $(\Phi, \Psi) \neq (0,0)$. To do this, let us first prove that $(u_n,v_n)$ does not vanish in a suitable Lebesgue space.

\begin{proposition}\label{estimatenorm}
Let $(u_{n},v_{n})$ be the $(PS)_{\omega}$-sequence obtained in Proposition \ref{existencePS}. Then,
$$
\liminf_{n\to\infty}\|(u_{n},v_{n})\|_{L^{2k+2} \times L^{2k+2}}^{2k+2} \neq 0.
$$
\end{proposition}
\begin{proof} By definition and \eqref{derivativeI} we deduce
\begin{equation}\label{limitenoconvergence}
\begin{split}
 \omega &= I(u_{n},v_{n}) - \frac{1}{2} I'(u_{n},v_{n})(u_{n},v_{n}) + o(1)\\
 & = k P(u_{n},v_{n}) + o(1).
\end{split}
\end{equation}
On the other hand, form  H\"older's inequality, there exists a positive constant $C$ such that
$$
0 < \left|  \displaystyle  k P(u_{n},v_{n}) \right| \leq C \|(u_{n},v_{n})\|_{L^{2k+2} \times L^{2k+2}}^{2k+2}.
$$
Hence, if $\liminf_{n\to\infty}\|(u_{n_j},v_{n_j})\|_{L^{2k+2} \times L^{2k+2}}^{2k+2} = 0$, we would have $\liminf_{n\to\infty}\displaystyle   P(u_{n_j},v_{n_j}) = 0 $. Taking the $\liminf$ in  \eqref{limitenoconvergence} we would obtain  $\omega = 0$, which is a contradiction.
\end{proof}

The rest of this section is devoted to prove a version of Lieb's translation lemma  (see \cite{lieb83}). Here we will follow the ideas presented in \cite{frank}. We start by observing that, up to a subsequence,
 Proposition \ref{estimatenorm} ensures the existence of a constant $C_{2k+2}>0$  such that
\begin{equation}\label{norm1}
\|(u_{n},v_{n})\|_{L^{2k+2} \times L^{2k+2}} \geq C_{2k+2}.
\end{equation}
Since $(u_{n},v_{n})$ is a bounded sequence in $H^{1}(\mathbb{R}) \times H^{1}( \mathbb{R})$, there exists a positive constant satisfying
\begin{equation}\label{norm2}
\|(u_{n},v_{n})\|_{L^{2} \times L^{2}} \leq C_{2}.
\end{equation}
In addition, by Sobolev's embedding  $H^{1}(\mathbb{R}) \hookrightarrow L^{r}(\mathbb{R})$, for every $r \in [2, \infty]$, there exists $r \in (2k+2, \infty)$ such that
\begin{equation}\label{norm3}
\|(u_{n},v_{n})\|_{L^{r} \times L^{r}} \leq C_{r}.
\end{equation}

Inequalities  \eqref{norm1}-\eqref{norm3} lead to a version of the  \textit{pqr Theorem}. 

\begin{theorem}[$pqr$ Theorem]\label{teopqr}
For any $0 < p < q < r < \infty$ and any constants $C_{p}$, $C_{q}$, $C_{r} > 0$, there are positive numbers $\epsilon$ and $\delta$ such that for any  $(f,g)$ satisfying
$$
\|(f,g)\|_{L^{p} \times L^{p}} \leq C_{p} \ ; \ \|(f,g)\|_{L^{q} \times L^{q}} \geq C_{q} \ \textrm{and} \ \|(f,g)\|_{L^{r} \times L^{r}} \leq C_{r},
$$
we have $\big(|\{  |f(x)| > \epsilon\} |  + |\{  |g(x)| > \epsilon \}|\big) \geq  \delta$.

In  other words, if the  $L^{p}$ and $L^{r}$ norms of a sequence $(f_{n},g_{n})$ are controlled from above and the $L^{q}$ norm is controlled from bellow   then this sequence cannot converge to zero in measure. 
\end{theorem}
\begin{proof}
The proof follows as in Lemma 3.2 of \cite{frank}.
\end{proof}

Now, before proving a version of Lieb's translation lemma, we need the following estimate for the $L^{2}(\mathbb{R})$ norm. 

\begin{lemma}\label{estimatesL2}
Let $r >0$. If $(u,v) \in H^{1}(\mathbb{R}) \times H^{1}(\mathbb{R}) $, then there are constants $C_{1}$, $C_{2} > 0$ such that
\begin{align}
\nonumber \|(u,v)\|^{2} \leq C_{1} \left( \sup_{y_{1} \in \mathbb{R}} |B_{r}(y_{1}) \cap \{ u \neq 0 \} |^{2} + \sup_{y_{2} \in \mathbb{R}} |B_{r}(y_{2}) \cap \{ v \neq 0 \} |^{2} \right)
\times \left[ \|\partial_{x}( u, v)\|^{2} + C_{2} r^{-2} \|(u,v)\|^{2} \right],
\end{align}
where $B_r(y)$ denotes the interval $(y-r,y+r)$.
\end{lemma}

\begin{proof} Fix a real-valued  function $\phi \in C_{c}^{\infty}(\mathbb{R})$  with  $supp \ \phi \subseteq B_{1}(0)$ and $\|\phi\|_{L^{2}}=1$. Define for each  $r>0$  and $y \in \mathbb{R}$ the function
$$
\phi_{r, y} (x) = r^{-\frac{1}{2}}\phi\left( \frac{x-y}{r}\right).
$$
A direct calculation gives
\begin{equation}\label{des2}
\displaystyle \int  |\phi_{r, y} (x)|^{2} dx =1
\end{equation}
and
\begin{equation}\label{des3}
\displaystyle \int  |\partial_{x} \phi_{r, y} (x)|^{2} dx =r^{-2}\|\partial_{x} \phi\|_{L^{2}}^{2}.
\end{equation}
In addition, since $\phi_{r, y} \in C_{c}^{\infty}(\mathbb{R})$ and $ u$, $v \in H^{1}(\mathbb{R})$ we may  compute
\begin{align*}
\displaystyle \int_{\mathbb{R}^{2}} &| \partial_{x} (\phi_{r, y}(x)u(x))|^{2} dy dx   + \int_{\mathbb{R}^{2}} | \partial_{x}(\phi_{r, y}(x)v(x))|^{2} dy dx \\
&= \displaystyle \int_{\mathbb{R}^{2}} |u(x)|^{2} |\partial_{x} \phi_{r, y} (x)|^{2} dy dx + \int_{\mathbb{R}^{2}} | \partial_{x}u(x)|^{2}|\phi_{r, y}(x)|^{2}| dy dx \\
& \ \ \  + \int_{\mathbb{R}^{2}} |v(x)|^{2}|\partial_{x}\phi_{r, y}(x)|^{2} dy dx + \int_{\mathbb{R}^{2}} | \partial_{x}v(x)|^{2}|\phi_{r, y}(x)|^{2}| dy dx \\
& \ \ \  + \displaystyle  \int_{\mathbb{R}^{2}} 2u(x) \partial_{x} \phi_{r, y}(x) \phi_{r, y}(x)\partial_{x}u(x) dy dx + \displaystyle  \int_{\mathbb{R}^{2}} 2v(x) \partial_{x} \phi_{r, y}(x) \phi_{r, y}(x) \partial_{x} v(x) dydx  .
\end{align*}
From Fubini's theorem and  identities  \eqref{des2} and \eqref{des3}, we obtain
\begin{align}\label{des4}
\nonumber \displaystyle \int_{\mathbb{R}^{2}} &  | \partial_{x}(\phi_{r, y}(x)u(x))|^{2} dy dx + \int_{\mathbb{R}^{2}} |\partial_{x}(\phi_{r, y}(x)v(x))|^{2} dy dx  = \displaystyle r^{-2}\|\partial_{x}\phi\|_{L^{2}}^{2} \|(u,v)\|^{2}  + \|\partial_{x}( u, v)\|^{2}\\
& + \displaystyle  \int_{\mathbb{R}^{2}} 2u(x)\partial_{x}\phi_{r, y}(x) \phi_{r, y}(x)\partial_{x} u(x) dy dx  + \displaystyle  \int_{\mathbb{R}^{2}} 2v(x) \partial_{x}\phi_{r, y}(x) \phi_{r, y}(x) \partial_{x} v(x) dy dx .
\end{align}

Next, by using H\"older and Young's inequalities we estimate the two integrals on the right-hand side of \eqref{des4} as follows
\begin{align*}
\displaystyle  \int_{\mathbb{R}^{2}}& 2u(x) \partial_{x} \phi_{r, y}(x) \phi_{r, y}(x) \partial_{x} u(x) dy dx  + \displaystyle  \int_{\mathbb{R}^{2}} 2v(x) \partial_{x}\phi_{r, y}(x) \phi_{r, y}(x)  \partial_{x}v(x) dy dx\\
& \leq  2r^{-1}\| \partial_{x} \phi \|_{L^{2}} \|u\|_{L^{2}} \|\partial_{x}u\|_{L^{2}} + 2r^{-1}\| \partial_{x} \phi \|_{L^{2}}\|v\|_{L^{2}} \|\partial_{x}v\|_{L^{2}}\\
& \leq   r^{-2}\| \partial_{x} \phi \|_{L^{2}}^{2} \|(u,v)\|^{2} + \| \partial_{x}(u,v)\|^{2}.
\end{align*}
Replacing in (\ref{des4}), we deduce
\begin{align}\label{des5}
\displaystyle \int_{\mathbb{R}^{2}} |\partial_{x} (\phi_{r, y} (x)u(x))|^{2} dy dx + \int_{\mathbb{R}^{2}} | \partial_{x}(\phi_{r, y}(x)v(x))|^{2} dy dx  
\leq 2 \displaystyle r^{-2}\|\partial_{x}\phi\|_{L^{2}}^{2} \|(u,v)\|^{2}  + \|\partial_{x}( u, v)\|^{2}.
\end{align}

On the other hand, since $supp \ (\phi_{r, y}u)$, $supp \ (\phi_{r, y}v) \subseteq B_{r}(y)$, from the Faber-Krahn inequality (see Proposition \ref{faber}) it follows that
%para $ (\phi_{r, y}u, \phi_{r,y}v)$.  %\textcolor{red}{d\'uvida expoente}
\begin{align*}
\int \left(|\partial_{x}(\phi_{r,y}(x)u(x))|^{2} +|\partial_{x}(\phi_{r,y}(x)v(x))|^{2}\right)dx & \geq C \left( | B_{r}(y) \cap \{u\neq 0\}|^{-2}\int |\phi_{r,y}(x)u(x)|^{2}dx \right. \\
& \ \ \ +  \left. | B_{r}(y) \cap \{v \neq 0\}|^{-2}\int |\phi_{r,y}(x)v(x)|^{2}dx \right)
\end{align*}
Combining this last relation with  \eqref{des5} we deduce
\begin{align*}
\|\partial_{x}(u,v)\|^{2}&  \geq  C \left( \int \left[| B_{r}(y) \cap \{u\neq 0\}|^{-2}\int |\phi_{r,y}(x)u(x)|^{2}dx\right]dy \right. \\
& \ \ \ \left.+ \int \left[ | B_{r}(y) \cap \{v \neq 0\}|^{-2}\int |\phi_{r,y}(x)v(x)|^{2}dx\right]dy \right) -2 r^{-2}\|\partial_{x}\phi\|_{L^{2}}^{2} \|(u,v)\|^{2}\\
& \geq C\left(\inf_{y_{1}}  | B_{r}(y_{1}) \cap \{u \neq 0\}|^{-2}\|u\|_{L^{2}}^{2} + \inf_{y_{2}}  |
B_{r}(y_{2}) \cap \{v \neq 0\}|^{-2}\|v\|_{L^{2}}^{2}\right)\\
&  \ \ \ -2 r^{-2}\|\partial_{x}\phi\|_{L^{2}}^{2} \|(u,v)\|^{2}\\
& \geq C \min \Big\{ \inf_{y_{1}}  | B_{r}(y_{1}) \cap \{u \neq 0\}|^{-2} ;\inf_{y_{2}}  |
B_{r}(y_{2}) \cap \{v \neq 0\}|^{-2} \Big\} \|(u,v)\|^{2} \\
& \ \ \ -2 r^{-2}\|\partial_{x}\phi\|_{L^{2}}^{2} \|(u,v)\|^{2}.
\end{align*}
This last inequality gives us the desired.
\end{proof}

Finally we have all necessary tools to establish the desired compactness  criterion for bounded sequences in  $H^{1}(\mathbb{R}) \times H^{1}(\mathbb{R})$.

\begin{lemma}[Lieb's translation Lemma]\label{lieb}
Let $(u_{n},v_{n}) \subset H^{1}(\mathbb{R}) \times H^{1}(\mathbb{R}) $ be a bounded sequence. If there are positive numbers $\epsilon$ and $\delta $ such that
\begin{equation}\label{hiplieb}
\big(|\{ |u_{n}(x)| > \epsilon   \} |+| \{ |v_{n}(x)| > \epsilon   \} |\big) \geq  \delta, \ \  \ n \ \in \ \mathbb{N},
\end{equation}
then  there is a  sequence $( y^{1}_{n}, y^{2}_{n})   \subset \mathbb{R} \times \mathbb{R}$ such that, up to a subsequence,
$$
(\widetilde{u}_{n}, \widetilde{v}_{n}) := (u_{n}(\cdot + y^{1}_{n}), v_{n} (\cdot + y^{2}_{n})) \rightharpoonup (\Phi, \Psi) \ \  \textrm{in} \ \  H^{1}(\mathbb{R}) \times H^{1}(\mathbb{R}), 
$$
where $(\Phi , \Psi) \neq (0,0)$.
\end{lemma}
\begin{proof}
Define $(w_{n}, z_{n}) = \left(\left(|u_{n}| - \frac{\epsilon}{2}\right)_{+},\left( |v_{n}| - \frac{\epsilon}{2}\right)_{+} \right)$, where $f_+$ denotes the positive part of $f$. Note that 
\begin{equation*}%\label{lieb1}
\|\partial_{x} (w_{n}, z_{n})\|^{2} \leq \| \partial_{x}(u_{n}, v_{n})\|^{2} \leq K,
\end{equation*}
for some positive constant $K$. Moreover, combining Proposition \ref{chebyshev} with \eqref{hiplieb} we obtain
\begin{equation}\label{lieb2}
\begin{split}
\displaystyle \|(w_{n},z_{n})\|^{2}  
& \geq \left(\frac{\epsilon}{2}\right)^{2} \left(| \{ |u_{n}| > \epsilon \} | + | \{ |v_{n}| > \epsilon \} | \right) \geq  \left(\frac{\epsilon}{2}\right)^{2} \delta > 0.
\end{split}
\end{equation}

An application of  Lemma \ref{estimatesL2} with $(u,v) = (w_{n},z_{n})$ and $r=1$ yields
\begin{align*}
\displaystyle\|(w_{n}, z_{n})\|^{2}&  \leq   C_{1}\left[  \| \partial_{x}(w_{n}, z_{n})\|^{2} + C_{2}  \|(w_{n},z_{n})\|^{2}  \right]   \\
& \quad \times  \left( \sup_{y_{1} \in \mathbb{R}} |B_{1}(y_{1}) \cap \{ w_{n} \neq 0 \} |^{2} + \sup_{y_{2} \in \mathbb{R}} |B_{1}(y_{2}) \cap \{z_{n} \neq 0 \} |^{2} \right)\\
& \leq  C_{1} \left( K + C_{2}\|(w_{n},z_{n})\|^{2} \right) \\
& \quad \times \displaystyle  \left( \sup_{y_{1} \in \mathbb{R}} | B_{1}(y_{1}) \cap \{w_{n} > 0\} |^{2} +  \sup_{y_{2} \in \mathbb{R}} | B_{1}(y_{2}) \cap \{z_{n} > 0\} |^{2} \right),
\end{align*}
where in the second inequality we used the fact that  $(w_{n},z_{n})$ is nonnegative.
Now, multiplying the last inequality by $\|(w_{n},z_{n})\|^{-2}$ and using \eqref{lieb2}, we have
\begin{align*}
1 
& \leq C_{3}  \left( \sup_{y_{1} \in \mathbb{R}} | B_{1}(y_{1}) \cap \{w_{n} > 0\} | +  \sup_{y_{2} \in \mathbb{R}} | B_{1}(y_{2}) \cap \{z_{n} > 0\} |\right)^{2} ,
\end{align*}
for some positive constant $C_{3}$ (depending on $K,\epsilon$ and $\delta$).
Thus, for each $n \in \mathbb{N}$, we may take $y^{1}_{n}$, $y^{2}_{n} \in \mathbb{R}$ such that
\begin{align*}
\frac{1}{2}  \leq C_{3}  \left( | B_{1}(y^{1}_{n}) \cap \{w_{n} > 0\} | +  | B_{1}(y^{2}_{n}) \cap \{z_{n} > 0\} | \right).
\end{align*}

Next, by using Chebyshev's inequality again we obtain 
\begin{align*}
\displaystyle \|(u_{n}, v_{n})\|_{L^{1}(B_{1}(y^{1}_{n})) \times L^{1}(B_{1}(y^{2}_{n }))} & \geq 
 \frac{\epsilon}{2}  \left( | B_{1}(y^{1}_{n}) \cap \{|u_{n}| > \frac{\epsilon}{2} \} | +  | B_{1}(y^{2}_{n}) \cap \{|v_{n}| > \frac{\epsilon}{2} \} | \right)\\
&   = 
\frac{\epsilon}{2}  \left( | B_{1}(y^{1}_{n}) \cap \{w_{n} > 0 \} | +  | B_{1}(y^{2}_{n}) \cap \{z_{n} > 0 \} | \right) \\
& \geq \frac{\epsilon}{2} \left(\frac{1}{2C_{3}}\right) = C_{4} > 0.
\end{align*}
Thus, defining $(\widetilde{u}_{n}(x), \widetilde{v}_{n}(x)) := ( u_{n} (x + y^{1}_{n}), v_{n}(x + y^{2}_{n}))$ this last inequality implies that
\begin{equation}\label{lieb3}
\displaystyle \| (\widetilde{u}_{n}, \widetilde{v}_{n})\|_{L^{1}(B_{1}(0)) \times L^{1}(B_{1}(0))} = \|(u_{n}, v_{n})\|_{L^{1}(B_{1}(y^{1}_{n})) \times L^{1}(B^{1}(y^{2}_{n}))} \geq C_{4}.
\end{equation}

Finally, since $ ( \widetilde{u}_{n},\widetilde{v}_{n}) $  is a bounded sequence in $H^{1}(\mathbb{R}) \times H^{1}(\mathbb{R})$, there exists $ (\Phi, \Psi) \in H^{1}(\mathbb{R}) \times H^{1}(\mathbb{R})$ such that, up to a subsequence, $(\widetilde{u}_{n}, \widetilde{v}_{n}) \rightharpoonup (\Phi,\Psi)$. Since weak convergence in turn implies  strong convergence  in $L_{loc}^{p}(\mathbb{R}) \times L_{loc}^{p}(\mathbb{R})$ for all $p \in [1, \infty)$ (see Theorem 8.6 in \cite{lieb1997analysis}), we deduce that $(\widetilde{u}_{n},\widetilde{v}_{n}) \rightarrow (\Phi,\Psi)$ in $L^{1}(B_{1}(0)) \times L^{1}(B_{1}(0))$. From (\ref{lieb3}), we then conclude that $(\Phi, \Psi) \neq (0,0)$.
\end{proof}

\begin{remark}
From Theorem \ref{teopqr}, the hypothesis \eqref{hiplieb} is satisfied by any bounded sequence in $H^{1}(\mathbb{R}) \times H^{1}(\mathbb{R})$ and unbounded in $L^{q}(\mathbb{R}) \times L^{q}(\mathbb{R})$ for some $q \in (2, \infty)$.
\end{remark}

\subsection{Existence of ground states}

The sequence $(u_{n}, v_{n})$ obtained in Proposition \ref{existencePS} satisfies \eqref{norm1}-\eqref{norm3}. Then, Proposition \ref{convegencederivative} and the $pqr$ theorem guarantees that such a sequence is in the assumptions of the Lieb translation lemma. Consequently, we obtain a subsequence such that, up to a translation, it converges weakly in $H^{1}(\mathbb{R}) \times H^{1}(\mathbb{R})$ to some point $(\Phi, \Psi) \neq (0,0)$. The idea now is to show that $(\Phi, \Psi)$ is indeed as ground state. To do so, we introduce the  \textit{Nehari} manifold
\begin{equation*}%\label{nehari}
\mathcal{N} = \{ (u,v) \in  H^{1}(\mathbb{R}) \times H^{1}(\mathbb{R})\setminus \{(0,0)\} \, : \, I'(u,v)(u,v) =0 \}.
\end{equation*}

\begin{remark}
From \eqref{derivativeI} and \eqref{functionh2},
$$
I'(u,v)(u,v)=\|(u,v)\|_{H^1\times H^1}^2-(2k+2)P(u,v),
$$
which immediately gives that $\mathcal{N} \subset \mathcal{P}$. In addition, if $(u,v)\in \mathcal{N}$ then
\begin{equation}\label{igual}
\|(u,v)\|_{H^1\times H^1}^2=(2k+2)P(u,v).
\end{equation}
\end{remark} 

\begin{lemma}\label{lemacriticalpoint}
The following statements hold.
\begin{itemize}
\item[(i)] $(0,0)$ is a strict local \textit{minimum} of  $I$;
\item[(ii)]$I(\ell u,\ell v) < 0$ for every $(u,v) \in \mathcal{N}$ and for $\ell$ sufficiently large.
\end{itemize}
\end{lemma}
\begin{proof} 
Part (i) follows immediately from \eqref{pa1}.
Now, observe that for any $\ell >0$,
\begin{equation*}
\begin{split}
I(\ell u,\ell v) &=  \displaystyle \frac{\ell^{2}}{2}\|(u,v)\|_{H^{1} \times H^{1} }^{2} -\ell^{2k+2}  P(u,v) . \\
\end{split}
\end{equation*}
Since $(u,v) \in \mathcal{N}\subset \mathcal{P}$, we have $- P(u,v) < 0$, so that (ii) follows.
\end{proof}

\begin{proposition}\label{lemacriticalindice}
For every $(u,v) \in \mathcal{P}$ there exists a unique number $\overline{\ell} >0$ such that $(\overline{\ell} u, \overline{\ell } v) \in \mathcal{N}$ and   $\displaystyle \max_{t \geq0} I(t u,t v)  = I(\overline{\ell} u,\overline{\ell} v)$.
\end{proposition}
\begin{proof}
Let
$h:  \left[0, \infty\right)  \rightarrow  \mathbb{R}$ be defined as $h(\ell)=I(\ell u,\ell v)$.
From \eqref{derivativeI}, we have
\begin{equation*}
\begin{split}
\displaystyle I'(\ell u,\ell v)(\ell u,\ell v) 
& \displaystyle = \ell \left( \ell \|(u,v)\|_{H^{1} \times H^{1} }^{2} - (2k+2) \ell^{2k+1} P(u,v) dx \right) \displaystyle = \ell h'(\ell).
\end{split}
\end{equation*}
It is simple matter to check that
\begin{equation}\label{lmaximo}
\displaystyle \overline{\ell} = \left( \frac{ \|(u,v)\|_{H^{1} \times H^{1} }^{2}}{\displaystyle (2k+2) P(u,v)} \right)^{\frac{1}{2k}}
\end{equation}
is the unique positive critical point of $h$. In addition since $\overline{ \ell }$ is clearly a maximum point we obtain the desired.
\end{proof}

Now we introduce the ``Nehari level'' as
\begin{equation}\label{neharilevel}
  \omega_{\mathcal{N}} = \displaystyle\inf_{\mathcal{N}} I(u,v). 
\end{equation}

Next result shows the mountain pass level and Nehari level are the same.
\begin{lemma}\label{idindices}
Let $\omega$ be defined in \eqref{mim3}. Under the above notation, there holds $\omega_{\mathcal{N}} =  \omega$.
\end{lemma}
\begin{proof}
Let us first prove that $\omega \geq \omega_{\mathcal{N}}$.
Let $\gamma \in \Gamma$. Since $t\mapsto I(\gamma(t)) $ is a continuous function on $[0,1]$ and $\gamma(0)=(0,0)$ is a strict local minimum of $I$ it follows that $I(\gamma(t)) > 0$ for small $t$. By continuity, $\gamma$ crosses $\mathcal{N}$, that is, there exists ${\ell_{0}} \in (0,1)$ such that  $\gamma({\ell_{0}}) \in \mathcal{N}$. Hence,
$$
\omega_{\mathcal{N}} \leq I(\gamma(\ell_{0}) \leq \max_{\ell \in [0,1]} I(\gamma ( \ell )).
$$
Since this inequality holds for any $\gamma\in\Gamma$ we have the desired.

Next we show that $\omega \leq \omega_{\mathcal{N}}$. Take any $(u,v)\in\mathcal{N}$.  From Lemma \ref{lemacriticalpoint} there exists $\ell$ sufficiently large such that $I(\ell u,\ell v)<0.$ Since $(u,v)\in \mathcal{N}\subset\mathcal{P}$ and $H$ is homogeneous of degree $2k+2$ it follows that $(\ell u,\ell v)\in\mathcal{P}$. Thus, from the proof of Proposition  \ref{lemacriticalindice} the number
\begin{equation}\label{lmaximo1}
\displaystyle \overline{\ell} = \left( \frac{ \|(u,v)\|_{H^{1} \times H^{1} }^{2}}{\displaystyle (2k+2)\ell^{2k} P(u,v)} \right)^{\frac{1}{2k}}
\end{equation}
is such that $(\overline{\ell} \ell u, \overline{\ell }\ell  v) \in \mathcal{N}$ and   $\displaystyle \max_{t \geq0} I(t\ell u,t \ell v)  = I(\overline{\ell}\ell u,\overline{\ell}\ell v)$. By defining $\gamma(t)=(t\ell u,t \ell v)$ we see that $\gamma\in\Gamma$ and
\begin{equation*}
\omega\leq \max_{t \in [0,1]} I(\gamma(t))\leq \max_{t \geq0} I(t\ell u,t \ell v)  = I(\overline{\ell}\ell u,\overline{\ell}\ell v).
\end{equation*}
But from \eqref{lmaximo1} and \eqref{igual} we deduce that $\overline{ \ell }\ell=1$, implying that $\omega\leq I(u,v)$. The proof is thus completed.
\end{proof}

In view of Lemma \ref{idindices} we are able to establish that the {infimum} of $I$ is indeed achieved. 

\begin{theorem}[Existence of ground states]\label{teodeexistenciags}
There exists at least one {ground state} solution  for the elliptic system \eqref{sistemaeliptico2}.
\end{theorem}
\begin{proof} Let $(u_{n},v_{n})$ be the $(PS)_{\omega}$-sequence provided by Proposition \ref{existencePS}. As we already said, in view of \eqref{norm1}-\eqref{norm3} and the \textit{pqr}  theorem  we can apply Lieb's translation lemma to obtain $(\Phi, \Psi) \neq (0,0)$ such that, up to a subsequence and translation, (see Theorem 8.6 and Corollary 8.7 in \cite{lieb1997analysis}),
\begin{equation*}%\label{convergencegs1}
(u_{n},v_{n}) \rightharpoonup (\Phi, \Psi) \ \  \mbox{in} \ \  H^{1}(\mathbb{R}) \times H^{1}(\mathbb{R}),
\end{equation*}
\begin{equation*}%\label{convergencegs2}
(u_{n},v_{n}) \longrightarrow (\Phi, \Psi) \ \  \mbox{q.t.p.}\;\mbox{in}\;\R,
\end{equation*}
and
\begin{equation}\label{convergencegs3}
(u_{n},v_{n}) \longrightarrow (\Phi, \Psi) \ \  \mbox{in} \ \ L^{q}_{loc}(\mathbb{R}) \times  L^{q}_{loc}(\mathbb{R}), \ \ \ q \ \in [1, \infty).
\end{equation}

Let us now show that $(\Phi, \Psi)$ is a ground state. Since $\omega=\omega_{\mathcal{N}}$, we need to establish the following.
\begin{itemize}
\item[1)]  $I'(\Phi, \Psi)=0$ 
\item[2)]  $\omega = I(\Phi, \Psi)$.
\end{itemize}
We split the proof in two steps. 
\vskip.2cm
\noindent \textbf{Step 1.} $I'(\Phi, \Psi) =0$.

By density, it is sufficient to prove that
\begin{equation*}%{convergenciapontocritico}
\left[I'(u_{n}, v_{n}) - I'(\Phi, \Psi) \right](w,z) \longrightarrow 0 \ , \ \mbox{for any} \ (w,z) \in C_{c}^{\infty}(\mathbb{R}) \times C_{c}^{\infty}(\mathbb{R}) .
\end{equation*}
To this end, note that by the weak convergence  the first two integrals in \eqref{derivativeI} applied to $(u,v) =(u_{n}, v_{n})$ satisfy
\begin{align*}
\int  (u_{n}w + v_{n}z) dx + \int [\partial_{x}u_{n} \partial_{x} w + \partial_{x} v_{n} \partial_{x}z] dx \longrightarrow  \int (\Phi w + \Psi z) dx + \int [\partial_{x} \Phi  \partial_{x} w + \partial_{x} \Psi \partial_{x} z]dx 
\end{align*}
Thus in order to conclude this step it suffices to prove that
$$
 \int \Big(f(u_{n}, v_{n})w + g(u_{n}, v_{n})z\Big) dx \longrightarrow  \int \Big(f(\Phi, \Psi) w + g(\Phi, \Psi) z \Big)dx,
$$
This last convergence follows once we establish (for instance) that
$$
\left| \int (u_{n}^{\alpha} v_{n}^{\beta}   - \Phi^{\alpha}\Psi^{\beta}) w dx \right| \longrightarrow 0.
$$
where $\alpha$ and $\beta$ are positive real numbers such that $\alpha + \beta = 2k+1$. Since $k \geq 2$, we can assume without loss of generality $\beta \geq  1$ and $\beta>\alpha$. Note we can rewrite
$$
u_{n}^{\alpha} v_{n}^{\beta} - \Phi^{\alpha}\Psi^{\beta} = u_{n}^{\alpha}(v_{n}^{\alpha} - \Psi^{\alpha})v_{n}^{\beta - \alpha} + v_{n}^{\beta - \alpha} (u_{n}^{\alpha} - \Phi^{\alpha})\Psi^{\alpha} + \Phi^{\alpha} \Psi^{\alpha} (v_{n}^{\beta -\alpha} - \Psi^{\beta - \alpha}).
$$ 
Taking $\Omega = supp \, w \cup supp \, z$, from  \eqref{convergencegs3}, we obtain 
\begin{align*}
\Big| \int  &  (u_{n}^{\alpha} v_{n}^{\beta}   - \Phi^{\alpha}\Psi^{\beta}) w dx \Big|  \leq \int_{\Omega}|u_{n}^{\alpha} v_{n}^{\beta}   - \Phi^{\alpha}\Psi^{\beta}| |w |dx \\
& \leq \int_{\Omega} \left[|u_{n}^{\alpha}||v_{n}^{\alpha} - \Psi^{\alpha}||v_{n}^{\beta - \alpha} |+ |v_{n}^{\beta - \alpha} ||u_{n}^{\alpha} - \Phi^{\alpha}||\Psi^{\alpha}| + |\Phi^{\alpha}|| \Psi^{\alpha} ||v_{n}^{\beta -\alpha} - \Psi^{\beta - \alpha}| \right] |w| dx \\
& \leq C \left[  \|(u_{n},v_{n}) - (\Phi, \Psi)\|_{L^{\alpha}(\Omega) \times L^{\alpha}(\Omega)} + \|(u_{n},v_{n}) - (\Phi, \Psi)\|_{L^{\beta - \alpha}(\Omega) \times L^{\beta - \alpha}(\Omega)} \right]  \longrightarrow 0,
\end{align*}
which is the desired.

\vskip.2cm
\noindent \textbf{Step 2.} $\omega = I(\Phi, \Psi)$.

From Step 1, $(\Phi, \Psi)$ is a critical point of $I$. In particular, $(\Phi, \Psi) \in \mathcal{N}$ and
\begin{equation*}%\label{ineqgs1}
\omega =\omega_{\mathcal{N}}= \inf_{\mathcal{N}}I(u,v) \leq I(\Phi, \Psi).
\end{equation*}

It remains to prove that $I(\Phi, \Psi) \leq \omega$. For this, we will show that  $(\widehat{u}_{n}, \widehat{v}_{n}):= (u_{n} -\Phi, v_{n} - \Psi)$ is a $(PS)_{\omega - I(\Phi, \Psi)}$-sequence. Proposition \ref{convegencederivative} then implies that  $\omega - I(\Phi, \Psi) \geq 0$. We need to prove that
\begin{align*}
& I(\widehat{u}_{n}, \widehat{v}_{n}) \longrightarrow \omega - I(\Phi, \Psi),\\
& I'(\widehat{u}_{n}, \widehat{v}_{n}) \longrightarrow 0 \ \ \ \mbox{in} \ \ \ H^{-1}(\mathbb{R}) \times H^{-1} (\mathbb{R}),
\end{align*}
or, equivalently,
\begin{align}
\label{equivconvergencepoitcritico1} & I(\widehat{u}_{n} + \Phi, \widehat{v}_{n} + \Psi) - I(\widehat{u}_{n}, \widehat{v}_{n}) \longrightarrow   I(\Phi, \Psi),\\
\label{equivconvergencepoitcritico2}& I'(\widehat{u}_{n} + \Phi, \widehat{v}_{n} + \Psi) - I'(\widehat{u}_{n}, \widehat{v}_{n}) \longrightarrow  I'(\Phi, \Psi) =0  \ \ \ \mbox{em} \ \ \ H^{-1}(\mathbb{R}) \times H^{-1}(\mathbb{R}).
\end{align}

Note that
\begin{equation}\label{intauxconvergence}
\begin{split}
 I(\widehat{u}_{n} + \Phi, \widehat{v}_{n} + \Psi) - I(\widehat{u}_{n}, \widehat{v}_{n})&  - I(\Phi, \Psi) =  \int[ \widehat{u}_{n} \Phi + \widehat{v}_{n} \Psi] dx + \int [\partial_{x}\widehat{u}_{n}\partial_{x}\Phi + \partial_{x}\widehat{v}_{n} \partial_{x}\Phi] dx \\
&  + \int H(\widehat{u}_{n}, \widehat{v}_{n}) dx + \int H(\Phi , \Psi) dx - \int H(\widehat{u}_{n} + \Phi , \widehat{v}_{n} + \Psi) dx . 
\end{split}
\end{equation}
Since $(\widehat{u}_{n}, \widehat{v}_{n}) \rightharpoonup (0,0)$ in $H^{1}(\mathbb{R}) \times H^{1}(\mathbb{R})$, the first two integrals on the right-hand side of \eqref{intauxconvergence} converge to zero.  In addition, after cancellation of the terms with opposite sign the remaining terms on the right-hand side of \eqref{intauxconvergence} are of the form
\begin{align}\label{integralconvergente}
\int \widehat{u}_{n}^{\alpha} \widehat{v}_{n}^{\beta}\Phi^{\tilde{\alpha}} \Psi^{\tilde{\beta}} dx,
\end{align}
where $\alpha$, $\widetilde{\alpha}$, $\beta$ and $\widetilde{\beta}$ are nonnegative  numbers with $\alpha + \tilde{ \alpha }  + \beta + \tilde{ \beta } = 2k+2 $. 

We claim that all integrals in \eqref{integralconvergente} converge to zero. To give a flavor of the proofs we consider only the case where
 $\alpha + \widetilde{\alpha} = 2k+2$ and $\alpha \cdot \widetilde{\alpha} \neq 0$ (the other terms converge to zero similarly).
 Assume first $\alpha =1$. Since $k\geq2$  then $\widetilde{\alpha} = 2k+1 \geq 5$. From Cauchy-Schwarz's inequality and Sobolev's embedding, we obtain
\begin{align*}
0 \leq \left| \int \widehat{u}_{n} \Phi^{2k+1} dx \right| \leq \|\Phi\|_{L^{\infty}}^{2k -1} \int |\widehat{u}^{n} \Phi|| \Phi |dx   \leq  C \|\widehat{u}_{n} \Phi\|_{L^{2}} \|\Phi\|_{L^{2}}.
\end{align*}
To conclude the claim in this case it suffices to show that
\begin{equation}\label{convergezero}
\|\widehat{u}_{n} \Phi\|_{L^{2}}^{2} \longrightarrow 0.
\end{equation}
In fact, for any $L >0$, 
\begin{align*}%\label{intauxconveg2}
 \nonumber 0 &\leq \|\widehat{u}_{n} \Phi\|_{L_{x}^{2}}^{2} 
 \leq \|\Phi\|_{L^{\infty}}^{2} \int_{-L}^{L} |\widehat{u}_{n}|^{2} dx + \left( \sup_{|x| \geq L} |\Phi(x)| \right)^{2} \int_{|x| \geq L} |\widehat{u}_{n}|^{2} dx =: A_n + B_n.
\end{align*}
From \eqref{convergencegs3} we have that $(\widehat{u}_{n})$ converges to zero in $L^{2}_{loc}$. So, $A_n$ converges to zero. Moreover, $B_n$ also converges to zero because $\lim_{L \rightarrow \infty} \sup_{|x| \geq L} |\Phi(x)| =0$, for any function $\Phi\in H^1(\R)$. 

Assume now $\alpha\geq2$.  Since  $\widetilde{\alpha} \geq 1$, Sobolev's embedding and H\"older's inequality imply
\begin{align*}
0 & \leq \left| \int \widehat{u}_{n}^{\alpha} \Phi^{\tilde{\alpha}}  dx  \right| \leq \|\widehat{u}_{n}\|_{L^{\infty}}^{\alpha - 2} \|\Phi\|_{L^{\infty}}^{ \tilde{\alpha} -1} \int |\widehat{u}_{n} \Phi| |\widehat{u}_{n}|dx 
%\leq C\|\widehat{u}_{n} \Phi\|_{L^{2}}^{2}\|\widehat{u}_{n}\|_{L^{2}} \\
 \leq C_{1} \|\widehat{u}_{n} \Phi\|_{L^{2}}^{2} 
\end{align*}
which reduces matter to \eqref{convergezero}. This establishes our claim  and \eqref{equivconvergencepoitcritico1} is proved.

To obtain \eqref{equivconvergencepoitcritico2} observe that for any $(w,z) \in C_{0}^{\infty}(\mathbb{R}) \times C_{0}^{\infty}(\mathbb{R})$ identity  \eqref{derivativeI} gives
\begin{align*}
&[I'(\widehat{u}_{n}, \widehat{v}_{n}) - I'(\widehat{u}_{n} - \Phi , \widehat{v}_{n} - \Psi) + I' (\Phi, \Psi)](w,z)\\
& = - \int [f(\widehat{u}_{n}, \widehat{v}_{n})w + g(\widehat{u}_{n}, \widehat{v}_{n})z]dx + \int [f(\widehat{u}_{n} + \Psi, \widehat{v}_{n} + \Psi) w + g(\widehat{u}_{n} + \Phi , \widehat{v}_{n} + \Psi ) z]dx \\
&\ \ \  - \int [f( \Phi,  \Psi) w + g( \Phi, \Psi ) z ] dx.
\end{align*}
After using the definitions of $f$ and $g$ the integral on the right-hand side of the last identity are of the form
$$
\int \widehat{u}_{n}^{\alpha} \Phi^{\tilde{\alpha}} \widehat{v}_{n}^{\beta} \Psi^{\tilde{\beta}} \theta dx
$$
with $\theta = w, z$ and $\alpha$, $\widetilde{\alpha}$, $\beta$ and $\widetilde{\beta}$ as in  \eqref{integralconvergente}. Using similar arguments as above,  \eqref{equivconvergencepoitcritico2} follows. 
This ends the proof that $(\widehat{u}_{n}, \widehat{v}_{n})$ is a $(PS)_{\omega - I(\Phi, \Psi)}$-sequence of $I$ and completes the proof of the theorem.
\end{proof}

As an immediate consequence of the existence of ground states we have the following.

\begin{corollary}\label{teosharpconst}
	For any $(u,v)\in \mathcal{P}$ we have
\begin{align*}%\label{inequalityGNtypeint2}
2  P(u,v)  \leq 	K_{opt}  \|(u,v)\|^{k+2}\|\partial_{x}(u,v)\|^{k},
\end{align*}
with the sharp constant $K_{opt} > 0$  given by
	\begin{equation}\label{kopt}
	K_{opt} = \frac{2}{k+2} \left(\frac{k+2}{k} \right)^{\frac{k}{2}}\frac{1}{\|(\Phi, \Psi)\|^{2k}},
	\end{equation}
	where $(\Phi, \Psi)$ is any  {ground state} solution of \eqref{sistemaeliptico2}.
\end{corollary}

\begin{proof}
This follows from \eqref{konstantoptima} and Proposition \ref{propIandsolution}.
\end{proof}

\section{Global well-posedness: proof of Theorem \ref{teogwpsystem}}\label{globalsec}

This section is devoted to prove Theorem \ref{teogwpsystem}. The main tool here is the sharp Gagliardo-Nirenberg inequality obtained in Corollary \ref{teosharpconst}.

Let $(u(t),v(t))$ be the solution of \eqref{PVIsystem} with initial data $(u_0,v_0)$. As in \eqref{idderivativeaux} we use the conservation laws \eqref{laws3} and \eqref{laws4} and Corollary \ref{teosharpconst} to write
\begin{equation}\label{ineqaux00}
\begin{split}
\|\partial_{x}(u(t), v(t))\|^{2}  & = 2E(u_{0}, v_{0}) +  2\mu P(u(t),v(t)) \\
&\leq 2E(u_{0}, v_{0}) +2 P(|u(t)|,|v(t)|) \\
&   \leq 2E(u_{0}, v_{0}) +  K_{opt} \|(u_{0},v_{0})\|^{k+2}\|\partial_{x}(u(t), v(t))\|^{k} .
\end{split}
\end{equation}
Now we split the proof into the  cases $k>2$ and $k=2$.\\

\noindent\textbf{Case $k > 2$.}\\
First, we note that under condition \eqref{hip2} we have $E(u_{0}, v_{0}) > 0$. In fact,  since  \eqref{ineqaux00} holds as long as the solution exists, by taking $t=0$ and using  \eqref{hip2}, we obtain
\begin{equation}\label{compt0}
\begin{split}
\|\partial_{x}(u_0, v_0)\|^{2} & \leq 2E(u_{0}, v_{0}) +  K_{opt} \|(u_{0},v_{0})\|^{k+2}\|\partial_{x}(u_0, v_0)\|^{k} \\
& < 2E(u_{0}, v_{0}) +  K_{opt} \|(\Phi , \Psi)\|^{k+2}\|\partial_{x}(\Phi, \Psi)\|^{k-2}\|\partial_{x}(u_0, v_0)\|^{2}.
\end{split}
\end{equation}
On the other hand, combining \eqref{Pohosaevid3} with \eqref{kopt} it follows that
\begin{align*}
K_{opt} \|(\Phi , \Psi)\|^{k+2}\|\partial_{x}(\Phi, \Psi)\|^{k-2} = \frac{2}{k}.
\end{align*} 
Since $k >2$, \eqref{compt0} then yields
$$
E(u_{0}, v_{0}) > \frac{k-2}{2k} \|\partial_{x}(u_0, v_0)\|^{2} >0.
$$

The idea now is to apply Lemma \ref{lemacontinuidade}. For this, we set
\begin{align*}
A =  2 E(u_{0}, v_{0})  > 0,\quad 
B = K_{opt}\|(u_{0},v_{0})\|^{k+2},\quad \mbox{and}\quad
G(t) = \|\partial_{x}(u(t), v(t))\|^{2}.
\end{align*}
Thus we can write \eqref{ineqaux00} as
\begin{align*}
A-G(t)+ B G^{\frac{k}{2}}(t)\geq0, \quad \mbox{for} \ t \in [0,T],
\end{align*}
with $T$ given by Theorem \ref{lwpsystem}. Thus, by defining $f(r)=A-r+Br^{m}$, $m=\frac{k}{2}$, we promptly see that $f(G(t))\geq0$, $t\in[0,T]$. Moreover, using \eqref{kopt}, in the notation of Lemma \ref{lemacontinuidade},
\[
\begin{split}
\gamma&=\left(\frac{k}{2}B\right)^{-\frac{2}{k-2}}\\
&=\left[\left(\frac{k+2}{k}\right)^{\frac{k}{2}-1}\frac{\|(u_0,v_0)\|^{k+2}}{\|(\Phi,\Psi)\|^{2k}} \right]^{-\frac{2}{k-2}}\\
&=\frac{k}{k+2}\frac{\|(\Phi,\Psi)\|^{\frac{4k}{k-2}}}{\|(u_0,v_0)\|^{\frac{2(k+2)}{k-2}}}.
\end{split}
\]
Hence,
\begin{align*}
G(0)<\gamma & \Leftrightarrow \|\partial_{x}(u_0,v_0)\|^2\|(u_0,v_0)\|^{\frac{2(k+2)}{k-2}}<\frac{k}{k+2}\|(\Phi,\Psi)\|^{\frac{4k}{k-2}}\\
& \Leftrightarrow \|\partial_{x}(u_0,v_0)\|^{k-2}\|(u_0,v_0)\|^{k-2}<\left(\frac{k}{k+2}\right)^\frac{k-2}{2}\|(\Phi,\Psi)\|^{2k}\\
& \Leftrightarrow \|\partial_{x}(u_0,v_0)\|^{k-2}\|(u_0,v_0)\|^{k-2}<\|\partial_{x}(\Phi,\Psi)\|^{k-2}\|(\Phi,\Psi)\|^{k+2},
\end{align*}
where in the last inequality we used \eqref{Pohosaevid3}. Thus, we see that $G(0)<\gamma$ is equivalent to \eqref{hip2}. Also, from \eqref{Pohosaevid3} and \eqref{Pohosaevid4} it is easily checked that
$$
E(\Phi,\Psi)=\frac{k-2}{2(k+2)}\|(\Phi,\Psi)\|^2.
$$
Therefore,
\begin{align*}
A<\left(1-\frac{1}{m}\right)\gamma &\Leftrightarrow E(u_0,v_0)\|(u_0,v_0)\|^{\frac{2(k+2)}{k-2}}<\frac{k-2}{2(k+2)}\|(\Phi,\Psi)\|^{\frac{4k}{k-2}}\\
& \Leftrightarrow E(u_0,v_0)\|(u_0,v_0)\|^{\frac{2(k+2)}{k-2}} < E(\Phi,\Psi)\|(\Phi,\Psi)\|^{\frac{2(k+2)}{k-2}}\\
& \Leftrightarrow E(u_0,v_0)^{k-2}M(u_0,v_0)^{k+2}<E(\Phi,\Psi)^{k-2}M(\Phi,\Psi)^{k+2},
\end{align*}
which means that $A<\left(1-\frac{1}{m}\right)\gamma$ is equivalent to \eqref{hip1}. As an application of Lemma \ref{lemacontinuidade} we deduce that $G(t)<\gamma$ which in turn is equivalent to \eqref{res2}. This completes the proof in the case $k>2$.\\

\noindent \textbf{Case $k =2$.} In this case, from \eqref{ineqaux00},
\begin{align*}
\|\partial_{x}(u(t), v(t))\|^{2}  \leq  2E(u_{0}, v_{0}) +  K_{opt} \|(u_{0},v_{0})\|^{4}\|\partial_{x}(u(t), v(t))\|^{2} .
\end{align*}
Thus, it suffices to require
\begin{align*}
K_{opt} \|(u_{0},v_{0})\|^{4} < 1.
\end{align*}
But from \eqref{kopt} with $k=2$,
\begin{align*}
K_{opt} \|(u_{0},v_{0})\|^{4} < 1\Leftrightarrow \frac{2}{4} \left(\frac{4}{2} \right)^{\frac{2}{2}}\frac{1}{\|(\Phi, \Psi)\|^{4}}\|(u_{0},v_{0})\|^{4} < 1 \Leftrightarrow \|(u_{0},v_{0})\| < \|(\Phi, \Psi)\|,
\end{align*}
which is the desired.

In both cases, we obtain a uniform bound for $\|\partial_{x}(u(t), v(t))\|$ and the proof of the theorem is completed.

\section{Characterization of the ground states: proof of Theorem \ref{char}} \label{secchar}

In this section we will give another characterization of the ground states. 
Introduce the functionals 
\begin{equation}\label{Sdef}
S(u,v):=\int (u_x^2+v_x^2+u^2+v^2)dx=\|(u,v)\|_{H^1\times H^1}^2
\end{equation}
and
\begin{equation*}
\widetilde{P}(u,v)=(2k+2)P(u,v)=(2k+2)\int H(u,v)dx.
\end{equation*}

For $\lambda>0$, consider the following minimization problem
\begin{equation}\label{Slambda}
S_\lambda=\inf\{S(u,v):\; (u,v)\in H^1(\R)\times H^1(\R) \;\mbox{with}\; \widetilde{P}(u,v)=\lambda\}.
\end{equation}
We will show that for a specific value of $\lambda$ the ground states of \eqref{sistemaeliptico2} are also solutions of \eqref{Slambda}. We start by noting that from \eqref{inequalityGNtypeint} we have $\Pt(u,v)\leq CS(u,v)^{k+1}$, which implies that $S_\lambda$ must be positive for any $\lambda>0$. In addition, the homogeneity of $S$ and $\Pt$ gives that
\begin{equation}\label{relS1}
S_\lambda=\lambda^{\frac{1}{k+1}}S_1.
\end{equation}

In what follows we set
$$
\lambda_1:=(S_1)^{\frac{k+1}{k}}.
$$

\begin{lemma}\label{equilam}
	Let $\omega_{\mathcal{N}}$ be the Nehari level introduced in \eqref{neharilevel}. Then,
	$$
	\lambda_1=\frac{2k+2}{k}\omega_{\mathcal{N}}.
	$$
\end{lemma}
\begin{proof}
Let $(u,v)$ be a ground state solution of \eqref{sistemaeliptico2}. In particular, we have $\omega_{\mathcal{N}}=I(u,v)$. Recall from \eqref{igual} we must have $S(u,v)=\Pt(u,v)$. Hence,
\begin{equation}\label{a1}
\begin{split}
\omega_{\mathcal{N}}& = I(u,v)=\frac{1}{2}S(u,v)-\frac{1}{2k+2}\Pt(u,v)=\frac{k}{2k+2}S(u,v).
\end{split}
\end{equation}
Define $(U,V)=\left(\frac{k}{\omega_{\mathcal{N}}(2k+2)}\right)^{\frac{1}{2k+2}}(u,v)$. From \eqref{a1}, we deduce
$$
\Pt(U,V)=\frac{k}{\omega_{\mathcal{N}}(2k+2)}\Pt(u,v)=\frac{k}{\omega_{\mathcal{N}}(2k+2)}S(u,v)=1
$$
and
$$
S_1\leq S(U,V)=\left(\frac{k}{\omega_{\mathcal{N}}(2k+2)}\right)^{\frac{1}{k+1}}S(u,v)=\left(\frac{k}{\omega_{\mathcal{N}}(2k+2)}\right)^{-\frac{k}{k+1}}.
$$
This last inequality yields $\lambda_1\leq \frac{2k+2}{k}\omega_{\mathcal{N}}$. To show the opposite inequality it suffices to prove that $ \left(\frac{\omega_{\mathcal{N}}(2k+2)}{k}\right)^{\frac{k}{k+1}}\leq S(z,w)$, for any $(z,w)$ satisfying $\Pt(z,w)=1$. To do so, from Proposition \ref{lemacriticalindice}, if we set $\overline{ \ell }^{2k}=S(z,w)$ (see \eqref{lmaximo}) then $(Z,W)=\overline{ \ell }(z,w)\in\mathcal{N}$. Consequently,
\[
\begin{split}
\omega_{\mathcal{N}}\leq I(Z,W)=\frac{1}{2}\overline{\ell }^2S(z,w)-\frac{1}{2k+2}\overline{ \ell }^{2k+2}=\frac{k}{2k+2}S(z,w)^\frac{k+1}{k},
\end{split}
\]
which immediately gives the desired.
\end{proof}

With the above lemma in hand we are able to give the following characterization of the ground states.

\begin{proposition}\label{propcar}
A pair $(u,v)\in H^1(\R)\times H^1(\R)$ is a ground state solution of \eqref{sistemaeliptico2} if and only if $S(u,v)=S_{\lambda_1}$ and $\Pt(u,v)=\lambda_1$.

In particular, the minimization problem \eqref{Slambda} with $\lambda=\lambda_1$ has at least one solution.
\end{proposition}
\begin{proof}
Let $(u,v)$ be a ground state solution of \eqref{sistemaeliptico2}. From \eqref{a1}, Lemma \ref{equilam}, and \eqref{relS1},
\begin{equation}\label{s1l1}
S(u,v)=\frac{2k+2}{k}\omega_{\mathcal{N}}=\lambda_1=\lambda_1^{\frac{1}{k+1}}S_1=S_{\lambda_1}
\end{equation}
and
$$
\Pt(u,v)=S(u,v)=\lambda_1.
$$

Now, take  $(u,v)\in H^1(\R)\times H^1(\R)$ satisfying $S(u,v)=S_{\lambda_1}$ and $\Pt(u,v)=\lambda_1$. Let us first show that $(u,v)$ is indeed a solution of \eqref{sistemaeliptico2}. From Lagrange's multiplier theorem there is a constant $\theta$ such that
$$
\int (u'w'+uw)dx=(k+1)\theta\int f(u,v)w\,dx,
$$
$$
\int (v'z'+vz)dx=(k+1)\theta\int g(u,v)z\,dx,
$$
for any $(w,z)\in H^1(\R)\times H^1(\R)$. Thus, we must show that $(k+1)\theta=1$. By taking $w=u$, $z=v$, and adding the above identities we deduce
$$
S(u,v)=(k+1)\theta \int [f(u,v)u+g(u,v)v]dx=(k+1)\theta (2k+2)\int H(u,v)dx=(k+1)\theta \Pt(u,v),
$$
that is, $S_{\lambda_1}=(k+1)\theta \lambda_1$. Since $\lambda_1=S_{\lambda_1}$ (see \eqref{s1l1}) we obtain the desired.

It remains to show that $(u,v)$ minimizes the functional $I$. Let $(z,w)$ be any solution of \eqref{sistemaeliptico2} and set $\xi=\Pt(z,w)$. Since $S(z,w)=\Pt(z,w)$ (see \eqref{Pohosaevid1}) we obtain
\begin{equation}\label{a2}
I(z,w)=\frac{1}{2}S(z,w)-\frac{1}{2k+2}\Pt(z,w)=\frac{k}{2k+2}\xi.
\end{equation}
Next define $(Z,W)=\left(\frac{\lambda_1}{\xi}\right)^{\frac{1}{2k+2}}(z,w)$. We have $\Pt(Z,W)=\lambda_1$ and
$$
\lambda_1^{\frac{1}{k+1}}S_1=S_{\lambda_1}=S(u,v)\leq S(Z,W)=\left(\frac{\lambda_1}{\xi}\right)^{\frac{1}{k+1}}S(z,w)=\lambda_1^{\frac{1}{k+1}}\xi^{\frac{k}{k+1}},
$$
implying that $\xi\geq (S_1)^{\frac{k+1}{k}}=\lambda_1$. Therefore, from \eqref{a2},
$$
I(u,v)=\frac{1}{2}\Pt(u,v)-\frac{1}{2k+2}\Pt(u,v)=\frac{k}{2k+2}\lambda_1\leq \frac{k}{2k+2}\xi=I(z,w).
$$
This completes the proof of the proposition.
\end{proof}

Finally we are in a position to prove Theorem \ref{char}

\begin{proof}[Proof of Theorem \ref{char}]
$(\Rightarrow)$ Assume first that $(u,v)$ is a nonnegative ground state. Let $U(x)=\sqrt{u(x)^2+v(x)^2}$ and fix $(x_0,y_0)\in Y$. From Proposition \ref{propcar} and the homogeneity of $F$,
\begin{equation}\label{a3}
\begin{split}
\lambda_1&=\Pt(u,v)=\int F(u(x),v(x))dx=\int F\left(\frac{1}{U(x)}(u(x),v(x))\right)U(x)^{2k+2}dx\\
& \leq \int F(x_0,y_0)U(x)^{2k+2}dx=\Pt(U(x)x_0,U(x)y_0).
\end{split}
\end{equation}
In addition, since $| U'|^2\leq | u'|^2+| v'|^2$ and $x_0^2+y_0^2=1$, we deduce
\begin{equation}\label{a4}
S(U(x)x_0,U(x)y_0)=\int (| U'(x)|^2+|U(x)|^2)dx\leq S(u,v).
\end{equation}
From the homogeneity of $\Pt$ and \eqref{a3} there is $t\in(0,1]$ such that $\lambda_1=\Pt(tU(x)x_0,tU(x)y_0)$. However, from \eqref{a4},
$$
S(tU(x)x_0,tU(x)y_0)=t^2S(U(x)x_0,U(x)y_0)\leq t^2S(u,v).
$$
Since $(u,v)$ is a minimum of $S$ restricted to $\Pt=\lambda_1$ (see Proposition \ref{propcar}) we must have $t=1$, implying that
$$
\lambda_1=\Pt(U(x)x_0,U(x)y_0) \quad \mbox{and}\quad S(U(x)x_0,U(x)y_0)=S(u,v)=S_{\lambda_1}.
$$
Another application of Proposition \ref{propcar} yields that $(U(x)x_0,U(x)y_0)$ is a ground state of \eqref{sistemaeliptico2}. Thus,
\begin{equation}\label{a5}
\begin{cases}
(U''-U)x_0+f(x_0,y_0)U^{2k+1}=0,\\
(U''-U)y_0+g(x_0,y_0)U^{2k+1}=0.
\end{cases}
\end{equation}
By multiplying the first and second equations in \eqref{a5} by $x_0$ and $y_0$, respectively, and adding the obtained equations, we get
$$
U''-U+[f(x_0,y_0)x_0+g(x_0,y_0)y_0]U^{2k+1}=0.
$$
Recalling \eqref{functionh2} and the definition of $F$, we finally deduce that $U$ must be a solution of
\begin{equation}\label{a6}
U''-U+F_{max}U^{2k+1}=0.
\end{equation}
Moreover, since
$$
I(U(x)x_0,U(x)y_0)=\frac{1}{2}\int [U'(x)^2+U(x)^2]dx-\frac{F(x_0,y_0)}{2k+2}\int U(x)^{2k+2}dx
$$
and $(U(x)x_0,U(x)y_0)$ is a ground state, it follows that $U$ is a ground state of \eqref{a6}.
Consequently, $(F_{max})^{\frac{1}{2k}}U$ is a ground state of \eqref{scalarg}. Recall that a ground state solution of
$$
h''-h+ah^{2k+1}=0, \qquad (a>0),
$$
is a solution that minimizes the action
\begin{equation}\label{I1}
I_a(h):=\frac{1}{2}\int [h'^2+h^2]dx-\frac{a}{2k+2}\int h^{2k+2}dx.
\end{equation}

 From the uniqueness of the ground state of \eqref{scalarg} we deduce (up to a translation)
\begin{equation}\label{a7}
(F_{max})^{\frac{1}{2k}}U=Q.
\end{equation}
Next, since $u(x)\geq0$ and $v(x)\geq0$, we may write $(u(x),v(x))=U(x)(z(x),w(x))$ with $z(x)^2+w(x)^2=1$ and $z(x)\geq0, w(x)\geq0$. Thus,
\[
\begin{split}
\int F_{max}U(x)^{2k+2}dx&=\int F(x_0,y_0)U(x)^{2k+2}dx=\Pt(U(x)x_0,U(x)y_0)=\Pt(u,v)\\
&=\int F(z(x),w(x))U(x)^{2k+2}dx,
\end{split}
\]
from which follows 
$$
\int (F_{max}-F(z(x),w(x)))U(x)^{2k+2}dx=0,
$$
and, consequently, $F(z(x),w(x))=F_{max}$ almost everywhere. This implies that $(z(x),w(x))=(F_{max})^{\frac{1}{2k}}(\alpha,\beta)$, where the point $(F_{max})^{\frac{1}{2k}}(\alpha,\beta)$ belongs to $Y$. Therefore, from \eqref{a7},
$$
(u,v)=U(z,w)=(F_{max})^{\frac{1}{2k}}(U\alpha,U\beta)\equiv (\alpha Q,\beta Q),
$$
as desired.

\noindent $(\Leftarrow)$ Assume now $(F_{max})^{\frac{1}{2k}}(\alpha,\beta)=:(x_0,y_0)\in Y$ and let
$
(u,v)=(\alpha Q,\beta Q),
$
where $Q$ is the ground state of \eqref{scalarg}.
We first claim that $(u,v)$ is a solution of \eqref{sistemaeliptico2}. Indeed, if either $\alpha=0$ or $\beta=0$ (equivalently $x_0=0$ or $y_0=0$) this is trivial because $F$ assumes the maximum value $a$ at the points $(1,0)$ and $(0,1)$. So, we may assume $x_0\neq0$ and $y_0\neq0$. In this case, from Lagrange's multiplier theorem there exists a constant $\theta$ such that
\begin{equation}\label{a8}
\begin{cases}
(k+1)f(x_0,y_0)=\theta x_0,\\
(k+1)g(x_0,y_0)=\theta y_0.
\end{cases}
\end{equation}
By multiplying the first equation in \eqref{a8} by $x_0$, the second by $y_0$ and adding the obtained equations we deduce that
\begin{equation}\label{a9}
\theta=(k+1)F(x_0,y_0)=(k+1)F_{max}.
\end{equation}
Hence, \eqref{a8} and \eqref{a9} imply
$$
\frac{f(\alpha,\beta)}{\alpha}=(F_{max})^{-1}\frac{f(x_0,y_0)}{x_0}=(F_{max})^{-1}\frac{\theta}{k+1}=1
$$
and
$$
\frac{g(\alpha,\beta)}{\beta}=(F_{max})^{-1}\frac{g(x_0,y_0)}{y_0}=(F_{max})^{-1}\frac{\theta}{k+1}=1.
$$
Consequently,
\[
\begin{cases}
(\alpha Q)''-\alpha Q +f(\alpha Q, \beta Q)=\alpha\left(Q''-Q+\frac{f(\alpha,\beta)}{\alpha}Q^{2k+1}\right)=0,\\
(\beta Q)''-\beta Q +g(\alpha Q, \beta Q)=\beta\left(Q''-Q+\frac{g(\alpha,\beta)}{\beta}Q^{2k+1}\right)=0,
\end{cases}
\]
which gives that $(\alpha Q,\beta Q)$ is a solution of \eqref{sistemaeliptico2}.

It remains to show that $(\alpha Q,\beta Q)$ is in fact a ground state. To prove this, first note that
\begin{equation}\label{a10}
I(\alpha Q,\beta Q)=(F_{max})^{-\frac{1}{k}}I_1(Q),
\end{equation}
where $I_1$ is given in \eqref{I1}.
On the other hand, if $(z,w)$ is any ground state, by using we have already proved, we must have $(z,w)=(\widetilde{\alpha}Q,\widetilde{\beta}Q)$ for some $(\widetilde{\alpha},\widetilde{\beta})$ satisfying $(F_{max})^{\frac{1}{2k}}(\widetilde{\alpha},\widetilde{\beta})\in Y$. Hence,
\begin{equation}\label{a11}
I(z,w)=I(\widetilde{\alpha}Q,\widetilde{\beta}Q)=(F_{max})^{-\frac{1}{k}}I_1(Q).
\end{equation}
By comparing \eqref{a10} and \eqref{a11} we then see that  $(\alpha Q,\beta Q)$ is in indeed a ground state. 
\end{proof}

An immediate consequence of the characterization of the ground states in Theorem \ref{char} is the following.

\begin{corollary}\label{cordecay}
The ground states of \eqref{sistemaeliptico2} are radially symmetric with an exponential decay at infinity. 
\end{corollary}

To illustrate an application of Theorem  \ref{char} we will consider two examples: one in the case $k=2$ and another one for $k>2$.

\begin{corollary}\label{cork=2}
Assume $k=2$ and suppose $a=c=d>0$ and $b>0$. Then, \eqref{sistemaeliptico2} has a unique ground state, which is given by
$$
(\alpha Q,\alpha Q), \qquad \alpha=\frac{1}{(4a+b)^{\frac{1}{4}}},
$$
where $Q$ is the ground state solution of 
$$
Q''-Q+Q^5=0.
$$
\end{corollary}
\begin{proof}
According to Theorem \ref{char} we need to find the maximum points of
$$
F(x,y)=a\Big(x^6+y^6+3(x^4y^2+x^2y^4)\Big)+2bx^3y^3,
$$
restricted to the set $Z=\{(x,y)\in\R^2: \, x^2+y^2=1, x\geq0, y\geq0\}$. To obtain the critical points, from Lagrange's multiplier theorem, we must find all point $(x,y,\theta)$ satisfying
\begin{equation}\label{a12}
\begin{cases}
3(ax^5+bx^2y^3+2ax^3y^2+axy^4)=\theta x,\\
3(ay^5+bx^3y^2+ax^4y+2ax^2y^3)=\theta y,\\
x^2+y^2=1, \; x\geq0,\; y\geq0.
\end{cases}
\end{equation}
Note that $(1,0,3a)$ and $(0,1,3a)$ are always solution of \eqref{a12}. Moreover, $F(0,1)=F(1,0)=a$. To find other possible critical points we then may assume $x\neq0$, $y\neq0$.
Now, dividing the first equation in \eqref{a12} by $x$, the second one by $y$ we see that
$$
ax^4+bxy^3+2ax^2y^2+ay^4=ay^4+bx^3y+ax^4+2ax^2y^2,
$$
or equivalently,
$$
bxy(y^2-x^2)=0.
$$
This implies that $x=y=\frac{1}{\sqrt2}$. As a consequence, $\left(\frac{1}{\sqrt2},\frac{1}{\sqrt2},\frac{3}{4}(4a+b)\right)$ is also a solution of \eqref{a12} with $F\left(\frac{1}{\sqrt2},\frac{1}{\sqrt2}\right)=a+\frac{b}{4}$.  Since $F\left(\frac{1}{\sqrt2},\frac{1}{\sqrt2}\right)>a=F(1,0)$ we are done.
\end{proof}

\begin{corollary}\label{cork>2}
Assume $k>2$, $a>0$, $b>0$, and $c=d=\gamma a$, with $\gamma\geq0$. In the case $\gamma>0$ assume also $k$ is even and $\gamma<2^{k-1}-k/2$. Let $Q$ be the ground state solution of \eqref{scalarg}. Then we have the following.
\begin{itemize}
	\item[(i)] If $a\left(2^k-1-\frac{2\gamma(k+1)}{k}\right)>b$ then the ground states of \eqref{sistemaeliptico2} are of the form
	$$
	(\alpha Q,0) \quad \mbox{and} \quad (0,\alpha Q), \qquad \alpha=a^{-\frac{1}{2k}}.
	$$
	\item[(ii)]  If $a\left(2^k-1-\frac{2\gamma(k+1)}{k}\right)=b$ then the ground states of \eqref{sistemaeliptico2} are of the form
	$$
	(\alpha Q,0),  \quad (0,\alpha Q),  \qquad \alpha=a^{-\frac{1}{2k}}, \quad 
	$$
	or
	$$
	(\alpha Q, \alpha Q), \qquad \alpha=\frac{a^{-\frac{1}{2k}}}{\sqrt2}.
	$$
	\item[(iii)] If $a\left(2^k-1-\frac{2\gamma(k+1)}{k}\right)<b$ then the ground states of \eqref{sistemaeliptico2} are of the form
	$$
	(\alpha Q,\alpha Q)  \qquad \alpha=\left(a+b+\frac{2a\gamma(k+1)}{k}\right)^{-\frac{1}{2k}}.
	$$
	In particular the ground state is unique in this case.
\end{itemize}
\end{corollary}
\begin{proof}
Following the ideas in Corollary \ref{cork=2} we need to find the critical points of
$$
F(x,y)=a(x^{2k+2}+y^{2k+2})+2bx^{k+1}y^{k+1} + \frac{2a \gamma(k+1)}{k} (x^{k+2}y^{k} + x^{k}y^{k+2} )
$$
restricted to the set $Z=\{(x,y)\in\R^2: \, x^2+y^2=1, x\geq0, y\geq0\}$. The Lagrange's multiplier theorem implies we must solve the system
\begin{equation}\label{a13}
\begin{cases}
(k+1)[a\, x^{2k+1} + b \, x^ky^{k+1}+   \frac{k+2}{k} a  \gamma \, x^{k+1}y^{k} + a  \gamma \, x^{k-1}y^{k+1} ]=\theta x,\\
(k+1)[a\, y^{2k+1} + b \, x^{k+1}y^{k} +  \frac{k+2}{k} a  \gamma \, x^{k}y^{k+1} + a  \gamma \, x^{k+2}y^{k-1} ]=\theta y,\\
x^2+y^2=1, \; x\geq0,\; y\geq0.
\end{cases}
\end{equation}
We first observe that  $(1,0,a(k+1))$ and $(0,1,a(k+1))$ are solutions of \eqref{a13} with $F(0,1)=F(1,0)=a$. To find the other solutions we may assume $x>0$ and $y>0$. As in \eqref{a9} we deduce that at  any critical point $(x_0,y_0)$ we must have
\begin{equation}\label{a14}
F(x_0,y_0)=\frac{\theta}{k+1}.
\end{equation}
On the other hand, dividing the first equation in \eqref{a13} by $x$ and comparing the result with \eqref{a14} we deduce that at any critical point $(x_0,y_0)$ of $F$ restricted to $Z$,
\begin{equation}\label{a15}
F(x_0,y_0)=ax_0^{2k}+bx_0^{k-1}y_0^{k+1} + \frac{a \gamma (k+2)}{k} x^{k}y^{k} + a \gamma x^{k-2}y^{k+2}.
\end{equation}

Next, by dividing the first equation in \eqref{a13} by $x$, the second one by $y$, we see that any critical point of $F$ we must satisfy
$$
ax^{2k}+bx^{k-1}y^{k+1} + a \gamma x^{k-2}y^{k+2} =ay^{2k}+bx^{k+1}y^{k-1} + a \gamma x^{k+2}y^{k-2},
$$
or, which is the same,
$$
a(x^{2k}-y^{2k})+b(x^{k-1}y^{k+1}-x^{k+1}y^{k-1}) + a \gamma (x^{k-2}y^{k+2}-x^{k+2}y^{k-2})=0.
$$
Since $y\neq0$ we may introduce the variable $r=\frac{x}{y}$. Thus the last identity reads as
$$
a(r^{2k}-1)-br^{k-1}(r^2-1) - a \gamma r^{k-2} (r^{4}-1)=0.
$$
Observing that $r^{2k}-1=(r^2-1)p(r)$, where $p(r)=r^{2k-2}+r^{2k-4}+\ldots+r^2+1$, we see that our task reduces to finding all positive solutions of
\begin{equation}\label{a16}
a(r^2-1)\Big(p(r) - \gamma (r^{2} + 1)r^{k-2}\Big) -br^{k-1}(r^2-1)=0.
\end{equation}
It is clear that $r=1$ is a solution of \eqref{a16}. This means that $\left(\frac{1}{\sqrt{2}},\frac{1}{\sqrt2}\right)$ is a critical point with $F\left(\frac{1}{\sqrt{2}},\frac{1}{\sqrt2}\right)=\left(a + b +  \frac{2 a \gamma (k+1)}{k}\right)\frac{1}{2^k}$. The other solutions of \eqref{a16} (if they exist) must satisfy 
\begin{equation}\label{a17}
ap(r)-br^{k-1} - a \gamma (r^{2}+1) r^{k-2}=0, \qquad (r>0).
\end{equation}
Now, let $r_0$ be any solution of \eqref{a17}. This means that $$(x_0,y_0)=\left(\frac{r_0}{\sqrt{1+r_0^2}},\frac{1}{\sqrt{1+r_0^2}}\right)$$
is a critical point of $F$ restricted to $Z$. We claim that $F(x_0,y_0)<a=F(1,0)$. Indeed, using \eqref{a15} and \eqref{a17} we have
\[
\begin{split}
F(x_0,y_0)&=\frac{ar_0^{2k}}{(1+r_0^2)^k}+\frac{br_0^{k-1}}{(1+r_0^2)^k} + \left( \frac{k+2}{k}\right)\frac{a \gamma r_0^{k}}{(1+r_0^2)^k} + \frac{a \gamma r_0^{k-2}}{(1+r_0^2)^k}\\
& = \frac{a(r_0^{2k}+p(r_0) )}{(1+r_0^2)^k}\ + \left(\frac{2}{k} \right)\frac{a \gamma r_{0}^{k}}{(1+r_0^2)^k},
\end{split}
\]
where in the last equality we have used that $r_0$ is a solution of \eqref{a17} to write $br_0^{k-1}=ap(r_0)-a\gamma r_0^{k-2}(r_0^2+1)$.
If $\gamma =0$ we have $r_0^{2k}+p(r_0)<(1+r_0^2)^k$ for any $r_0>0$ and $k > 2$; so $F(x_0,y_0) < a$, as claimed. On the other hand, if $\gamma > 0$, it suffices that
\begin{equation}\label{a18}
r_0^{2k}+p(r_0) +\frac{2 \gamma r_{0}^{k}}{k} <(1+r_0^2)^k.
\end{equation}
Using the definition of $p$, expanding the right-hand side of \eqref{a18} and using that $k$ is even we see that it suffices
\begin{equation}\label{a19}
1+\frac{2\gamma}{k}<\frac{k}{((k/2)!)^2}.
\end{equation}
But observing that right-hand side of \eqref{a19} is grater than $\frac{2^k}{k}$ we deduce that it suffices to impose the condition $1+\frac{2\gamma}{k}<\frac{2^k}{k}$, which holds in view of our assumption on $\gamma$. The claim is thus proved.

  As a consequence, the maximum of $F$ restricted to $Z$ may occur only at the points $(0,1)$, $(1,0)$, or $\left(\frac{1}{\sqrt{2}},\frac{1}{\sqrt2}\right)$. By comparing the maximum value of $F$ at these points and using Theorem \ref{char} we complete the proof.
\end{proof}

\section{Instability of ground states for $k>2$}\label{secinst}

In this section we assume $k>2$ and restrict our attention to the system
\begin{equation}\label{partsys}
\begin{cases}
\partial_{t}u + \partial_{x}^{3}u+  \partial_{x}( H_{u}(u,v)) =0,\\
\partial_{t}v + \partial_{x}^{3}v +  \partial_{x} (H_{v}(u,v)) =0, 
\end{cases}
\end{equation}
with $H$ given according to Corollary \ref{cork>2}.
%\begin{equation}\label{newH}
%H(u,v)=\frac{a}{2k+2} \left( u^{2k+2} + v^{2k+2}\right)+ \frac{b}{k+1} (uv)^{k+1}, \quad a>0,b>0.
%\end{equation}
%In particular $H$ has the form \eqref{functionHint} with $c=d=0$. Thus, our goal is to prove the instability of the ground states obtained in Corollary \ref{cork>2}.
Recall that a solitary wave for \eqref{partsys} is a solution of the form $(\phi(x-\omega t),\psi(x-\omega t))$ with $\phi$ and $\psi$ having a suitable decay at infinity. By replacing this ansatz  in \eqref{partsys} and integrating once we obtain
\begin{equation}\label{b1}
\begin{cases}
\phi'' - \omega\phi + H_u(\phi, \psi )= 0,\\
\psi '' - \omega\psi +  H_v(\phi, \psi )= 0,
\end{cases}
\end{equation}
which reduces to system \eqref{sistemaeliptico2} if $\omega=1$.

Our first result concerns the existence of solutions for \eqref{b1} for any $\omega>0$.

\begin{lemma}\label{solwaves}
	Let $(\phi_1,\psi_1)$ be any ground state of \eqref{sistemaeliptico2} according to Corollary \ref{cork>2}. 
Then, system \eqref{b1} has a smooth curve of ground state solutions,
$$
\omega\in(0,+\infty)\mapsto (\phi_{\omega},\psi_{\omega})\in H^\infty(\R)\times H^\infty(\R),
$$
which agrees with $(\phi_1,\psi_1)$ at $\omega=1$.
\end{lemma}
\begin{proof}
This follows immediately by setting
$$
\phi_{\omega}(x)=\omega^{\frac{1}{2k}}\phi_1(\sqrt{\omega}x),\quad \psi_{\omega}(x)=\omega^{\frac{1}{2k}}\psi_1(\sqrt{\omega}x), \qquad \omega>0,
$$
and using that $H_u$ and $H_v$ are homogeneous functions of degree $2k+1$.
\end{proof}

Next let us recall the definition of orbital stability/instability.

\begin{definition}
We say that the solitary wave $(\phi_{\omega},\psi_{\omega})$ is stable in $H^1(\R)\times H^1(\R)$ if for any $\varepsilon>0$ there exists $\delta>0$ such that if $(u_0,v_0)\in H^1(\R)\times H^1(\R)$ satisfies $\|(u_0,v_0)-(\phi_{\omega},\psi_{\omega})\|_{H^1\times H^1}<\delta$ then the corresponding solution of \eqref{partsys} with initial data $(u_0,v_0)$ exists globally and satisfies
$$
\inf_{r\in\R}\|(u(t),v(t))-(\phi_{\omega}(\cdot+r),\psi_{\omega}(\cdot+r))\|_{H^1\times H^1}<\varepsilon
$$
for any $t\geq0$. Otherwise, we say that $(\phi_{\omega},\psi_{\omega})$ is unstable in $H^1(\R)\times H^1(\R)$.
\end{definition}

We will show that the solitary waves in Lemma \ref{solwaves} are unstable in $H^1(\R)\times H^1(\R)$. To do so, we define $G(u,v)=E(u,v)+\omega M(u,v)$ with $E$ and $M$ given in \eqref{laws3} and \eqref{laws4}. From \eqref{b1} we immediately see that $(\phi_{\omega},\psi_{\omega})$ is a critical point of $G$, that is,
\begin{equation}\label{critical}
G'(\phi_{\omega},\psi_{\omega})=(0,0).
\end{equation}
In addition, the linearization of $G$ around $(\phi_{\omega},\psi_{\omega})$ is the operator
\begin{equation}\label{L}
\mathcal{L}:=G''(\phi_\omega,\psi_{\omega})=
\begin{pmatrix}
-\partial_{x}^2+\omega & 0\\
0 & -\partial_{x}^2+\omega
\end{pmatrix}
-\begin{pmatrix}
H_{uu}(\phi_{\omega},\psi_{\omega}) & H_{uv}(\phi_{\omega},\psi_{\omega})\\
H_{uv}(\phi_{\omega},\psi_{\omega}) & H_{vv}(\phi_{\omega},\psi_{\omega})
\end{pmatrix}
\end{equation}

From Corollary \ref{cork>2} we know that $(\phi_{\omega},\psi_{\omega})=(\alpha Q_\omega,\beta Q_\omega)$, where $Q_\omega$ is the (unique) ground state solution of 
\begin{equation}\label{Qomega}
-Q''+\omega Q-Q^{2k+1}=0.
\end{equation}
Thus recalling that $F(x,y)=(2k+2)H(x,y)$ we obtain
\begin{equation}\label{L1}
\mathcal{L}=
\begin{pmatrix}
-\partial_{x}^2+\omega & 0\\
0 & -\partial_{x}^2+\omega
\end{pmatrix}
-\frac{(F_{max})^{-1}}{2k+2}\begin{pmatrix}
F_{xx}(x_0,y_0) & F_{xy}(x_0,y_0)\\
F_{xy}(x_0,y_0) & F_{yy}(x_0,y_0)
\end{pmatrix}Q_\omega^{2k},
\end{equation}
where $(x_0,y_0)$ is a maximum point of $F$, that is, according to Corollary \ref{cork>2}, $(x_0,y_0)$ is either $(1,0)$, $(0,1)$, or $\left(\frac{1}{\sqrt{2}},\frac{1}{\sqrt2}\right)$.

Now we have the following result.

\begin{theorem}\label{insta}
Assume that $\mathcal{L}$ has a unique negative eigenvalue which is simple. Assume also that zero is a simple eigenvalue and the rest of the spectrum is positive and bounded away from zero. Then, the solitary wave $(\phi_{\omega},\psi_{\omega})$ is unstable in $H^1(\R)\times H^1(\R)$ provided that $\Lambda''(\omega)<0$, where $\Lambda(\omega)=E(\phi_{\omega},\psi_{\omega})+\omega M(\phi_{\omega},\psi_{\omega})$.
\end{theorem}
\begin{proof}
The proof follows the same ideas as in Theorem 6.2 in \cite{aam}, which in turn is an extension to systems of the results in \cite{bona87}. So we will omit the details.  We just highlight that usually we need some strong decay at infinity of the solitary waves. Here this is not an issue because from Corollary \ref{cordecay} our solitary waves has an exponential decay.
\end{proof}

With Theorem \ref{insta} in hand we are able to prove the following.

\begin{theorem}\label{insta1}
Assume $k>2$. The solitary waves in Lemma \ref{solwaves} are unstable in $H^1(\R)\times H^1(\R)$ for any $\omega>0$.
\end{theorem}
\begin{proof}
	We will use Theorem \ref{insta}.
First let us study the spectrum of the operator $\mathcal{L}$. We will consider only the case when $(x_0,y_0)=\left(\frac{1}{\sqrt{2}},\frac{1}{\sqrt2}\right)$. The cases $(x_0,y_0)=(1,0)$ or $(x_0,y_0)=(0,1)$ are simpler.

By taking the derivative with respect to $x$ in \eqref{b1} we promptly obtain that $( Q_\omega', Q_\omega')$ belongs to the kernel of $\mathcal{L}$. Assume now $(u,v)$ is an eigenfunction of $\mathcal{L}$ associated to the eigenvalue $\lambda$. Thus,
\begin{equation}\label{b2}
\begin{cases}
-u''+\omega u-\dfrac{(F_{max})^{-1}}{2k+2}\left( F_{xx}(x_0,y_0)u+F_{xy}(x_0,y_0)v\right)Q_\omega^{2k}=\lambda u,\\
-v''+\omega v-\dfrac{(F_{max})^{-1}}{2k+2}\left( F_{xy}(x_0,y_0)u+F_{yy}(x_0,y_0)v\right)Q_\omega^{2k}=\lambda v.
\end{cases}
\end{equation}
Subtracting and adding the equations in \eqref{b2} we obtain
\begin{equation}\label{b3}
\begin{cases}
\begin{aligned}
-(u-v)''+\omega (u-v)-\dfrac{(F_{max})^{-1}}{2k+2}\Big(&(F_{xx}(x_0,y_0)-F_{xy}(x_0,y_0))u\\
&+(F_{xy}(x_0,y_0)-F_{yy}(x_0,y_0))v\Big)Q_\omega^{2k}=\lambda (u-v),
\end{aligned}\\
\begin{aligned}
-(u+v)''+\omega (u+v)-\dfrac{(F_{max})^{-1}}{2k+2}\Big(& (F_{xx}(x_0,y_0)+F_{xy}(x_0,y_0))u\\
&+(F_{xy}(x_0,y_0)+F_{yy}(x_0,y_0))v\Big)Q_\omega^{2k}=\lambda(u+v).
\end{aligned}
\end{cases}
\end{equation}
Using the definition of $F$ we deduce that

\begin{align*}
\dfrac{(F_{max})^{-1}}{2k+2}(F_{xx}(x_0,y_0)-F_{xy}(x_0,y_0))& =-\dfrac{(F_{max})^{-1}}{2k+2}(F_{xy}(x_0,y_0)-F_{yy}(x_0,y_0))\\
& =\left(a + b +  \frac{2 a \gamma (k+1)}{k}\right)^{-1}\left((2k+1)a-b - \dfrac{2a\gamma (k-1)}{k} \right)
\end{align*}
and
\begin{align*}
\dfrac{(F_{max})^{-1}}{2k+2} (F_{xx}(x_0,y_0)+F_{xy}(x_0,y_0)) & =\dfrac{(F_{max})^{-1}}{2k+2}(F_{xy}(x_0,y_0)+F_{yy}(x_0,y_0))=2k+1
\end{align*}
Thus, \eqref{b3} reduces to
\begin{equation}\label{b4}
\begin{cases}
-(u-v)''+\omega (u-v)-\left(a + b +  \frac{2 a \gamma (k+1)}{k}\right)^{-1}\left((2k+1)a-b - \dfrac{2a\gamma (k-1)}{k} \right)Q_\omega^{2k}(u-v)=\lambda (u-v),\\
-(u+v)''+\omega (u+v)-(2k+1)Q_\omega^{2k}(u+v)=\lambda (u+v).
\end{cases}
\end{equation}
Now we introduce the operators $\mathcal{L}_1=-\partial_{x}^2+\omega-(2k+1)Q_\omega^{2k}$ and $\mathcal{L}_2=-\partial_{x}^2+\omega-Q_\omega^{2k}$. From \eqref{Qomega} we see that $\mathcal{L}_1(Q_\omega')=0$ and $\mathcal{L}_2(Q_\omega)=0$. It is well known that $Q_\omega$ is given by
$$
Q_\omega(x)=\left((k+1)\omega\ \mathrm{sech}^2(k\sqrt{\omega}x)\right)^{\frac{1}{2k}}
$$
from which we obtain that $Q_{\omega}'$ has only one zero on the whole line. In particular, it follows from Sturm-Liouville theory that $\mathcal{L}_1$ has a unique negative eigenvalue, zero is a simple eigenvalue and the rest of the spectrum is positive and bounded away from zero (see, for instance, \cite[Theorem B.61]{angulobook}). Also, since $Q_\omega$ has no zeros on the whole line, it follows that zero is the first eigenvalue of $\mathcal{L}_2$ (see, for instance, \cite[Theorem B.59]{angulobook}) and the rest of the spectrum is bounded away from zero. In addition, recalling the relations between $a$ and $b$ in Corollary \ref{cork>2} and that we are assuming $(x_0,y_0)=\left(\frac{1}{\sqrt{2}},\frac{1}{\sqrt2}\right)$ we get
$$
(2k+1)a-b - \dfrac{2a\gamma (k-1)}{k}  < a + b +  \frac{2 a \gamma (k+1)}{k}.
$$
Thus comparing with $\mathcal{L}_2$ we infer that the first eigenvalue of the operator $$-\partial_{x}^2+\omega-\left(a + b +  \frac{2 a \gamma (k+1)}{k}\right)^{-1}\left((2k+1)a-b - \dfrac{2a\gamma (k-1)}{k} \right)Q_\omega^{2k}$$ must be positive and the rest of the spectrum is positive and bounded away from zero. Thus, the negative and null eigenvalues come only from $\mathcal{L}_1$. Putting all these information together and using \eqref{b4} we finally deduce that $\mathcal{L}$ has a unique negative eigenvalue, its kernel is one-dimensional and the rest of the spectrum is positive and bounded away from zero.

In order to conclude the proof of the theorem it remains to establish that $\Lambda''(\omega)<0$. But from \eqref{critical} we obtain $\Lambda'(\omega)=  M(\phi_{\omega},\psi_\omega)$. Since
\[
\begin{split}
M(\phi_{\omega},\psi_\omega)&=M(\alpha Q_\omega, \alpha Q_\omega)\\
&=\frac{(F_{max})^{-\frac{1}{k}}}{2}M(Q_\omega,Q_\omega)\\
&=\frac{(F_{max})^{-\frac{1}{k}}}{2}\frac{(k+1)^{\frac{1}{k}}}{k}\omega^{\frac{1}{k}-\frac{1}{2}}\int_{\R}\mathrm{sech}^{\frac{2}{k}}(x)dx.
\end{split}
\]
Consequently,
\begin{equation}\label{eqfinal}
\Lambda''(\omega)=\frac{(F_{max})^{-\frac{1}{k}}}{2}\frac{(k+1)^{\frac{1}{k}}}{k}\left(\frac{1}{k}-\frac{1}{2}\right)\omega^{\frac{1}{k}-\frac{3}{2}}\int_{\R}\mathrm{sech}^{\frac{2}{k}}(x)dx<0,
\end{equation}
because $k>2$. The proof of the theorem is thus completed.
\end{proof}

\begin{remark}
The approach presented in this section is not restricted to the case of ground states given in Corollary \ref{cork>2}. Actually, Theorem \ref{insta} is still true if, more generally, $H$ is given in \eqref{functionHint}.  Also, note that our calculations in \eqref{eqfinal} depends only on the characterization of the ground states in Theorem \ref{char}. So, the main difficulty in proving the instability in the general  case (with $k>2$) consists in establishing the spectral properties of the linearized operator in \eqref{L1}.

Note also that in the case $k=2$ (critical case) we obtain $\Lambda''(\omega)=0$. So that we are unable to conclude the stability/instability of the traveling waves.
\end{remark}

\subsection*{Acknowledgment}
This work is part of the Ph.D. Thesis of the first author,
which was concluded at IMECC-UNICAMP. The first author acknowledges the financial
support from Capes/Brazil and CNPq/Brazil. The second author is partially supported by CNPq/Brazil grants 402849/2016-7 and 303098/2016-3.

\bibliographystyle{acm}  
%\makeatletter
%\renewcommand\@biblabel[1]{{\parbox{0.8cm}{[#1]}}}
%\makeatother
\bibliography{reference}  %%% Remove comment to use the external .bib file (using bibtex).
%%% and comment out the ``thebibliography'' section.
%\printbibliography

\end{document}